\documentclass{amsart} 

\usepackage{amssymb}
\usepackage{amsmath}
\usepackage{amsthm}

\usepackage{enumerate} 

\theoremstyle{plain}
 \newtheorem{theorem}{Theorem}[section]
 \newtheorem{main-theorem}{Theorem}
 \newtheorem{corollary*}[main-theorem]{Corollary}
 \newtheorem{lemma}[theorem]{Lemma}
 
 \newtheorem{proposition}[theorem]{Proposition}

\theoremstyle{definition}
 \newtheorem{remark}[theorem]{Remark}
 \newtheorem{example}[theorem]{Example}
\newtheorem{definition}[theorem]{Definition}

\renewcommand{\mod}{\operatorname{mod}}

\newcommand{\rad}{\operatorname{rad}}
\newcommand{\soc}{\operatorname{soc}}

\newcommand{\Hom}{\operatorname{Hom}}
\newcommand{\Ker}{\operatorname{Ker}}
\newcommand{\T}{\operatorname{T}}

\newcommand{\op}{\operatorname{op}}

\newcommand{\charact}{\operatorname{char}}

\newcommand{\intt}{\operatorname{int}}

\newcommand{\vf}{\varphi}
\newcommand{\ve}{\varepsilon}
\newcommand{\bA}{\mathbb{A}}

\newcommand{\bN}{\mathbb{N}}

\newcommand{\bR}{\mathbb{R}}
\newcommand{\bS}{\mathbb{S}}
\newcommand{\bT}{\mathbb{T}}

\newcommand{\cO}{\mathcal{O}}

\newcommand{\cT}{\mathcal{T}}

\newcommand{\ba}{\bar{\alpha}}

\newcommand{\overbar}[1]{\mkern 5mu\overline{\mkern-5mu#1\mkern-5mu}\mkern 5mu}

\usepackage{etex} 

\usepackage{dsfont}         
\usepackage[h]{esvect}

\usepackage[matrix,arrow,curve,frame,arc]{xy}

\usepackage{pgf}
\usepackage{tikz}
\usepackage{xcolor}

\usetikzlibrary{arrows,shapes,automata,backgrounds,petri,calc}

\newcommand{\tikzAngleOfLine}{\tikz@AngleOfLine}
\def\tikz@AngleOfLine(#1)(#2)#3{%
\pgfmathanglebetweenpoints{%
\pgfpointanchor{#1}{center}}{%
\pgfpointanchor{#2}{center}}
\pgfmathsetmacro{#3}{\pgfmathresult}%
}

\begin{document}

\title{From Brauer graph algebras to biserial weighted surface algebras}

{\def\thefootnote{}
\footnote{The authors gratefully acknowledge support from the research grant
DEC-2011/02/A/ST1/00216 of the National Science Center Poland.}
}

\author[K. Edrmann]{Karin Erdmann}
\address[Karin Erdmann]{Mathematical Institute,
   Oxford University,
   ROQ, Oxford OX2 6GG,
   United Kingdom}
\email{erdmann@maths.ox.ac.uk}

\author[A. Skowro\'nski]{Andrzej Skowro\'nski}
\address[Andrzej Skowro\'nski]{Faculty of Mathematics and Computer Science,
   Nicolaus Copernicus University,
   Chopina~12/18,
   87-100 Toru\'n,
   Poland}
\email{skowron@mat.uni.torun.pl}

\begin{abstract}
We prove that the  class of Brauer graph
algebras coincides with the class of indecomposable
idempotent algebras of  biserial weighted surface
algebras. These algebras are associated to  
triangulated surfaces with arbitrarily oriented
triangles, investigated recently in \cite{ESk3} and \cite{ESk4}.
Moreover, we prove that Brauer graph algebras
are idempotent algebras of periodic weighted
surface algebras, investigated in \cite{ESk3} and \cite{ESk5}.

\bigskip

\noindent
\textit{Keywords:}
Brauer graph algebra,
Weighted surface algebra,
Biserial weighted surface algebra,
Symmetric algebra, 
Special biserial algebra,
Tame algebra,
Periodic algebra,
Quiver combinatorics
 
\noindent
\textit{2010 MSC:}
05E99, 16G20, 16G70, 20C20

\subjclass[2010]{05E99, 16G20, 16G70, 20C20}
\end{abstract}

\maketitle


\section{Introduction and the main results}\label{sec:intro}

Throughout this paper, $K$ will denote a fixed algebraically closed field.
By an algebra we mean an associative, finite-dimensional $K$-algebra
with an identity.
For an algebra $A$, we denote by $\mod A$ the category of
finite-dimensional right $A$-modules
and by $D$ the standard duality $\Hom_K(-,K)$ on $\mod A$.
An algebra $A$ is called \emph{self-injective}
if $A_A$ is an injective module, or equivalently,
the projective modules in $\mod A$ are injective.
Two self-injective algebras $A$ and $B$ are said
to be \emph{socle equivalent} if the quotient algebras
$A / \soc (A)$ and $B / \soc (B)$ are isomorphic.
Symmetric algebras are an important class of  self-injective algebras. 
An algebra 
 $A$ is symmetric if  there exists
an associative, non-degenerate, symmetric, $K$-bilinear form
$(-,-): A \times A \to K$. Classical examples of symmetric
algebras include in particular, 
 blocks of group algebras of finite groups and
Hecke algebras of finite Coxeter groups.
In fact, any algebra $A$ is the quotient algebra
of its trivial extension algebra $\T(A) = A \ltimes D(A)$,
which is a symmetric algebra.
By general theory, if $e$ is an idempotent of a symmetric algebra $A$,
then the idempotent algebra $e A e$ also is a symmetric algebra.

Brauer graph algebras play a prominent role in the representation theory of tame
symmetric algebras. 
Originally, R. Brauer
introduced the Brauer tree, which led to the description of
blocks of group algebras of finite groups of
finite representation type, and they 
are the basis for their classification up to  Morita equivalence 
\cite{Da,J,Ku},
see also \cite{Al}. 
Relaxing the condition on the
characteristic of the field, one gets Brauer tree algebras, and these
 occurred in the Morita equivalence classification of
 symmetric algebras of Dynkin type
$\bA_n$
\cite{GR,Rie}.
If one allows  arbitrary multiplicities, and also an arbitrary graph instead
of just a tree, one obtains Brauer graph algebras. 
These occurred in the
classification of  symmetric algebras of Euclidean type
$\widetilde{\bA}_n$
\cite{BS}.
It was shown in
\cite{Ro}
(see also \cite{Sch})
that the class of Brauer graph algebras coincides
with the class of symmetric special biserial algebras.
Symmetric special biserial algebras occurred also
in the Gelfand-Ponomarev classification of singular
Harish-Chandra modules over the Lorentz group
\cite{GP},
and as well in the context of 
 restricted Lie algebras,
or more generally infinitesimal group schemes,
\cite{FS1,FS2},
and in classifications of tame Hecke algebras
\cite{AIP,AP,EN}.
There are also results on derived  equivalence classifications of
Brauer graph algebras, and on the connection to Jacobian algebras
of quivers with potential, we refer to 
\cite{Ai,De,Ka,MS,MH,Ric1,Sch}.

We recall the definition of a Brauer graph algebra,
following \cite{Ro}, see also \cite{Sch}.  
A \emph{Brauer graph}
is a finite connected graph $\Gamma$,
with at least one edge
(possibly with loops and multiple edges)
such that for each vertex
$v$ of $\Gamma$, there is a cyclic ordering of the edges 
adjacent to $v$, and there is
a multiplicity $e(v)$ which is a positive integer.
Given a Brauer graph $\Gamma$, 
one defines the associated
Brauer quiver $Q_{\Gamma}$  as follows:
\begin{itemize}
 \item
  the vertices $Q_{\Gamma}$ are the edges of $\Gamma$;
 \item
  there is an arrow $i \to j$ in $Q_{\Gamma}$ if and only if
  $j$ is the consecutive edge of $i$ in the cyclic ordering
  of edges adjacent to  a vertex $v$ of $\Gamma$.
\end{itemize}

In this case we say that the arrow $i\to j$ is attached to $v$.
The quiver $Q_{\Gamma}$ is 2-regular (see Section~\ref{sec:bisalg}). 
Recall that a quiver is
2-regular if 
every vertex is the source and target 
of exactly two arrows.
Any 2-regular quiver has a canonical involution $({-})$ on the arrows,
namely if $\alpha$  is an arrow the $\bar{\alpha}$ is the other arrow
starting at the same vertex as $\alpha$.

The associated Brauer graph algebra $B_{\Gamma}$  is a quotient algebra
of $KQ_{\Gamma}$.
The cyclic ordering of the edges adjacent to a vertex $v$ of $\Gamma$ translates to 
a cyclic permutation of the arrows in $Q_{\Gamma}$, and if $\alpha$ is an arrow in this cycle,
we denote vertex $v$ by $v(\alpha)$.  
Let $C_{\alpha}$ be the product of the arrows in the cycle, in
the given order, starting with $\alpha$, this is an element in $KQ_{\Gamma}$. 
The associated \emph{Brauer graph algebra} $B_{\Gamma}$
is defined to be  $K Q_{\Gamma} / I_{\Gamma}$,
where $I_{\Gamma}$ is the ideal in the path algebra
$K Q_{\Gamma}$  generated by
the elements:
\begin{enumerate}[(1)]
 \item
  all paths $\alpha \beta$ of length $2$ in $Q_{\Gamma}$
  which are not subpaths of $C_{\alpha}$,
 \item
  $C_{\alpha}^{e(v(\alpha))} - C_{\bar{\alpha}}^{e(v(\bar{\alpha}))}$,
  for all arrows $\alpha$ of $Q_{\Gamma}$.
\end{enumerate}

\medskip
In \cite{ESk3}  and \cite{ESk4} we introduced and studied   
biserial weighted surface algebras, motivated by 
tame blocks of group algebras of finite groups.
Given a triangulation $T$ of a 2-dimensional real compact surface, 
with or without boundary,
and an orientation $\vv{T}$ of triangles in $T$, 
there is a natural way to define a quiver $Q(S, \vv{T})$. 
We  showed that these  quivers
have an algebraic description:  
they are precisely what we called triangulation quivers.
A triangulation quiver is a pair $(Q, f)$ where $Q$ is a 2-regular quiver, 
and $f$ is a permutation of arrows of order $3$  such that
 $t(\alpha)= s(f(\alpha))$ for each arrow $\alpha$ of $Q$.
A \emph{biserial weighted  surface algebra}  $B(S, \vv{T}, m_{\bullet})$ 
is then explicitly given  by the quiver $Q(S, \vv{T})$ and relations,
depending on a weight function $m_{\bullet}$, 
and if described using the triangulation quiver, 
we get a \emph{biserial weighted triangulation algebra}
$B(Q, f,m_{\bullet})$ (see Section~\ref{sec:bisalg}).

Algebras of generalized dihedral type 
(see \cite[Theorem 1]{ESk4}) 
which contain blocks with dihedral defect groups, 
turned out to be (up to socle deformation) idempotent algebras 
of biserial weighted surface algebras, 
for very specific idempotents.
Biserial weighted surface algebras
belong to the class of  Brauer graph algebras.
It is therefore  a natural question to ask which other Brauer graph algebras
occur as idempotent algebras of biserial weighted surface algebras. 
This is answered by our first main result.

\begin{main-theorem}
\label{th:main1}
Let $A$ be a basic, indecomposable, finite-dimensional
$K$-algebra 
over an algebraically closed field $K$
of dimension at least $2$.
Then the following statements are equivalent:
\begin{enumerate}[(i)]
  \item
    $A$ is a Brauer graph algebra.
  \item
    $A$ is isomorphic to the idempotent algebra $e B e$
    for a biserial weighted surface algebra $B$ 
    and an idempotent $e$ of $B$.
\end{enumerate}
\end{main-theorem}

The main ingredient for this is Theorem~\ref{th:4.1}. 
This
gives a canonical construction, which we call $*$-construction. 
A byproduct of
the proof of Theorem~\ref{th:main1} is the  following fact.

\begin{corollary*}
\label{cor:main}
Let $A$ be a Brauer graph algebra
over an algebraically closed field $K$.
Then $A$ is isomorphic to the idempotent algebra $e B e$
of a biserial weighted surface algebra 
$B = B(S,\vv{T},m_{\bullet})$,
for a surface $S$ without boundary, a triangulation $T$ of $S$
without self-folded triangles, and an idempotent $e$ of $B$.
\end{corollary*}

Moreover, we can adapt the $*$-construction to algebras
socle equivalent to Brauer graph algebras, 
and  prove an  analog for the main part of Theorem~\ref{th:main1}:

\begin{main-theorem}
\label{th:main3}
Let $A$ be a symmetric algebra over an algebraically
closed field $K$ which is socle equivalent 
but not isomorphic to a Brauer
graph algebra,  
and assume 
the Grothendieck group $K_0(A)$ 
has rank at least $2$. Then
\begin{enumerate}[(i)]
 \item
  $\charact(K)=2$, and 
	\item
$A$ is isomorphic to an idempotent algebra
$\bar{e} \bar{B} \bar{e}$, where
$\bar{B}$ is a socle deformed biserial weighted surface algebra 
$\bar{B}  = B(S,\vv{T},m_{\bullet},b_{\bullet})$.
Here $S$ is  a surface  with boundary, $T$ is a triangulation  of $S$
without self-folded triangles, and  $b_{\bullet}$ is a  border function.
\end{enumerate}

\end{main-theorem}

Recall  that an algebra $A$ is called \emph{periodic}
if it is periodic with respect to action of the syzygy
operator $\Omega_{A^e}$ in the module category
$\mod A^e$, where
$A^e = A^{\op} \otimes_K A$ is its enveloping algebra.
If  $A$ is a periodic algebra of period $n$
then all  indecomposable non-projective right $A$-modules are periodic
of period dividing $n$, with respect to 
the syzygy operator $\Omega_A$ in $\mod A$.
Periodic algebras are self-injective,
and have  connections with group theory,
topology, singularity theory and cluster algebras.
In \cite{ESk3}  and \cite{ESk5} we introduced and studied   
weighted surface algebras $\Lambda(S, \vv{T}, m_{\bullet},c_{\bullet})$,
which are tame, symmetric, and we showed that
they are  
(with one exception) 
periodic algebras of period $4$. 
They are
defined by the quiver $Q(S, \vv{T})$ and explicitly  given relations,
depending on a weight function $m_{\bullet}$ 
and a parameter function $c_{\bullet}$ 
(see Section~\ref{sec:proof4}).
Most biserial weighted surface algebras occur as
geometric degenerations of these periodic weighted surface algebras.

Our third  main result connects  Brauer graph algebras
with  a large class of periodic weighted surface algebras.

\begin{main-theorem}
\label{th:main4}
Let $A$ be a Brauer graph algebra
over an algebraically closed field $K$.
Then $A$ is isomorphic to an idempotent algebra $e \Lambda e$
of a periodic weighted surface algebra 
$\Lambda = \Lambda(S,\vv{T},m_{\bullet},c_{\bullet})$,
for a surface $S$ without boundary, a triangulation $T$ of $S$
without self-folded triangles, and an idempotent $e$ of $\Lambda$.
\end{main-theorem}

There are many idempotent
algebras of weighted surface algebras which are neither
Brauer graph algebras nor periodic algebras.
We give an example at the end of Section~\ref{sec:proof4}.

\bigskip

This paper is organized as follows.
In Section~\ref{sec:bisalg}
we recall basic facts on special biserial algebras
and show that 
Brauer graph algebras, 
symmetric special biserial algebras, and 
 symmetric algebras associated to 
weighted biserial quivers are  essentially the same.
In Section~\ref{sec:bisweight} 
we introduce biserial weighted surface algebras 
and present their basic properties.
In Section~\ref{sec:proof}
we prove 
Theorem~\ref{th:main1}. 
This contains an algorithmic construction
which may be of independent interest.
Sections~\ref{sec:proof3} and \ref{sec:proof4} 
contain the proofs of
of Theorems \ref{th:main3} and \ref{th:main4}, 
and related material.
In the final Section~\ref{sec:diagram}
we present a diagram showing the relations between
the main classes of algebras occurring in the paper.

For general background on the relevant representation theory
we refer to the books 
\cite{ASS,E5,SS,SY}, 
and 
we refer to \cite{E5,ESk1}
for the representation theory
of arbitrary self-injective special biserial algebras.

\section{Special biserial algebras}\label{sec:bisalg}

A \emph{quiver} is a quadruple $Q = (Q_0, Q_1, s, t)$
consisting of a finite set $Q_0$ of vertices,
a finite set $Q_1$ of arrows,
and two maps $s,t : Q_1 \to Q_0$ which associate
to each arrow $\alpha \in Q_1$ its source $s(\alpha) \in Q_0$
and  its target $t(\alpha) \in Q_0$.
We denote by $K Q$ the path algebra of $Q$ over $K$
whose underlying $K$-vector space has as its basis
the set of all paths in $Q$ of length $\geq 0$, and
by $R_Q$ the arrow ideal of $K Q$ generated by all paths 
in $Q$ of length $\geq 1$.
An ideal $I$ in $K Q$ is said to be \emph{admissible}
if there exists $m \geq 2$ such that
$R_Q^m \subseteq I \subseteq R_Q^2$.
If $I$ is an admissible ideal in $K Q$, then
the quotient algebra $K Q/I$ is called
a \emph{bound quiver algebra}, and is a finite-dimensional
basic $K$-algebra.
Moreover, $K Q/I$ is indecomposable if and only if
$Q$ is connected.
Every basic, indecomposable, finite-dimensional
$K$-algebra $A$ has a bound quiver presentation
$A \cong K Q/I$, where $Q = Q_A$ is the \emph{Gabriel
quiver} of $A$ and $I$ is an admissible ideal in $K Q$.
For a bound quiver algebra $A = KQ/I$, we denote by $e_i$,
$i \in Q_0$, the associated complete set of pairwise
orthogonal primitive idempotents of $A$. Then the modules 
$S_i = e_i A/e_i \rad A$ (respectively, $P_i = e_i A$),
$i \in Q_0$, form a  complete family of pairwise
non-isomorphic simple modules (respectively, indecomposable
projective modules) in $\mod A$.

Following \cite{SW}, an algebra $A$ is said to be
\emph{special biserial} if $A$ is isomorphic
to a bound quiver algebra $K Q/I$, where the bound
quiver $(Q,I)$ satisfies the following conditions:
\begin{enumerate}[(a)]
 \item[(a)]
  each vertex of $Q$ is a source and target of at most two arrows,
 \item[(b)]
  for any arrow $\alpha$ in $Q$ there are at most
  one arrow $\beta$ and at most one arrow $\gamma$
  with $\alpha \beta \notin I$ and $\gamma \alpha \notin I$.
\end{enumerate}

Background on special biserial algebras may be found for example in 
\cite{BR,E5,PS,SW,WW}.
Perhaps most important is the following, which  has been proved by 
Wald and Waschb\"usch
in \cite{WW} (see also \cite{BR,DS3} for alternative
proofs).

\begin{proposition}
\label{prop:2.1}
Every special biserial algebra is tame.
\end{proposition}

If a special biserial algebra is in addition symmetric, 
there is a more convenient description. 
We propose the concept of  
a (\emph{weighted}) \emph{biserial quiver algebra}, 
which we will now
define. Later, 
in Theorem \ref{th:2.6} we will show that these algebras
are precisely special biserial symmetric algebras.

\begin{definition}\label{definition:2.2} 
A \emph{biserial quiver} is a pair $(Q,f)$,
where $Q = (Q_0,Q_1,s,t)$ is a finite connected quiver
and $f : Q_1 \to Q_1$ is a permutation 
of the arrows of $Q$ satisfying the following conditions:
\begin{enumerate}[(a)]
 \item  $Q$ is 2-regular, that is
  every vertex of $Q$ is the source and target 
  of exactly two arrows,
 \item
  for each arrow $\alpha \in Q_1$ 
  we have $s(f(\alpha)) = t(\alpha)$.
\end{enumerate}

Let $(Q,f)$ be a biserial quiver.
We  obtain another permutation
$g : Q_1 \to Q_1$ defined by
 $g(\alpha) = \overbar{f(\alpha)}$
for any $\alpha \in Q_1$, so that $f(\alpha)$ and
$g(\alpha)$ are the arrows starting at $t(\alpha)$.
Let  $\cO(\alpha)$ be the $g$-orbit of an arrow  
$\alpha$,  and set
$n_{\alpha} = n_{\cO(\alpha)} = |\cO(\alpha)|$.
We denote by $\cO(g)$ the set
of all $g$-orbits in $Q_1$.
A function
\[
  m_{\bullet} : \cO(g) \to \bN^* = \bN \setminus \{0\}
\]
is said to be a \emph{weight function} of $(Q,f)$.
We write briefly $m_{\alpha} = m_{\cO(\alpha)}$
for $\alpha \in Q_1$.
The multiplicity function $m_{\bullet}$
taking only value $1$ is said to be \emph{trivial}.
For any arrow $\alpha \in Q_1$, we single out 
the oriented cycle
\[
  B_{\alpha} = \Big( \alpha g(\alpha) \dots g^{n_{\alpha}-1}(\alpha)\Big)^{m_{\alpha}}
\]
of length $m_{\alpha} n_{\alpha}$.
The triple $(Q,f,m_{\bullet})$ is said to be a
(\emph{weighted}) \emph{biserial quiver}.

The associated \emph{biserial quiver algebra} $B = B(Q, f, m_{\bullet})$ 
is defined as follows. 
It is the quotient algebra
\[
  B(Q,f,m_{\bullet})
   = K Q / J(Q,f,m_{\bullet}),
\]
where $J(Q,f,m_{\bullet})$
is the ideal of the path algebra $KQ$ 
generated by the following  elements:
\begin{enumerate}[(1)]
 \item
  $\alpha f({\alpha})$,
  for all arrows $\alpha \in Q_1$,
 \item
  $B_{\alpha}
   - B_{\bar{\alpha}}$,
  for all arrows $\alpha \in Q_1$.
\end{enumerate}
We assume that $Q$ is not the quiver 
with one vertex and two loops
$\alpha$ and $\ba$ such that 
$\alpha = B_{\alpha}$ and $B_{\ba} = \ba$
are equal in $B$, that is
we exclude the 2-dimensional algebra 
isomorphic to $K[X]/(X^2)$.
Assume 
$m_{\alpha} n_{\alpha} = 1$, so that 
$\alpha = B_{\alpha}$ and $B_{\ba}$ are equal in $B$. 
By the above assumption, $B_{\ba}$ lies in the square 
of the radical of the algebra. 
Then $\alpha$ is not an arrow in the Gabriel quiver $Q_B$ of $B$,
and we call it a \emph{virtual loop}.
\end{definition}

The following  describes basic
properties of (weighted) biserial quiver algebras.

\begin{proposition}
\label{prop:2.3}
Let $(Q,f,m_{\bullet})$ be a weighted biserial quiver
and
$B = B(Q,f,m_{\bullet})$.
Then ${B}$ is a basic, indecomposable, 
finite-dimensional symmetric special biserial algebra
with $\dim_K {B} = \sum_{\cO \in \cO(g)} m_{\cO} n_{\cO}^2$.
\end{proposition}

\begin{proof} 
It follows from the definition 
that $B$ is the special biserial bound quiver 
algebra $K Q_B / I_B$, where $Q_B$ is obtained 
from $Q$ by removing all virtual loops, 
and where $I_B = J(Q,f,m_{\bullet}) \cap K Q_B$.
Let $i$ be a vertex of $Q$ and $\alpha, \bar{\alpha}$ 
the two arrows starting at $i$. 
Then the indecomposable projective $B$-module 
$P_i = e_i B$ has a basis given by $e_i$ together with
all initial proper subwords of $B_{\alpha}$ and $B_{\bar{\alpha}}$, and
and $B_{\alpha} (= B_{\bar{\alpha}})$, and hence
$\dim_K P_i = m_{\alpha} n_{\alpha} + m_{\bar{\alpha}} n_{\bar{\alpha}}$.
Note also that the union of these bases gives
a basis of $B$ consisting of paths in $Q$. 
We deduce that
$\dim_K {B} = \sum_{\cO \in \cO(g)} m_{\cO} n_{\cO}^2$.
As well, the indecomposable projective module $P_i$
has simple socle generated by  
$B_{\alpha} ( = B_{\bar{\alpha}})$.
We define a  symmetrizing $K$-linear form
$\varphi : B \to K$ as follows. If  $u$ is a path in $Q$ which belongs
to the above basis, we set 
$\varphi(u) = 1$ if $u = B_{\alpha}$ for an arrow $\alpha \in Q_1$,
and $\varphi(u) = 0$ otherwise.
Then $\varphi(a b) = \varphi(b a)$ for all elements $a,b \in B$
and $\Ker \varphi$ does not contain any non-zero one-sided
ideal of $B$, and consequently $B$ is a symmetric algebra
(see \cite[Theorem~IV.2.2]{SY}).
\end{proof}

\medskip

We wish to 
compare Brauer graph algebras and biserial quiver algebras. For this
we start analyzing  the combinatorial data.
Let $Q$ be a connected 2-regular quiver. 
We call a permutation $g$ of the arrows
of $Q$ \emph{admissible} if for every arrow $\alpha$ 
we have $t(\alpha)= s(g(\alpha))$. 
That is,  the arrows along a cycle of $g$ can be concatenated in $Q$.
The multiplicity function of a Brauer graph $\Gamma$
taking only value $1$ is said to be \emph{trivial}.
\medskip

\begin{lemma}\label{lem:2.4} 
There is  a bijection between Brauer graphs $\Gamma$ 
with trivial multiplicity function
and pairs $(Q, g)$ where $Q$ is a connected 2-regular quiver, 
and $g$ is an admissible
permutation of the arrows of $Q$.
\end{lemma}

\begin{proof}
(1) \ Given $\Gamma$, we take the quiver $Q=Q_{\Gamma}$, as defined
in the introduction. 

\smallskip

(1a) \ We show  that  $Q_{\Gamma}$ is 2-regular. 
Take an edge $i$ of $\Gamma$, it is adjacent 
to vertices $v, w$ (which may be equal).
If $v\neq w$ then the edge $i$ occurs both in the cyclic ordering 
around $v$ and of $w$, so there are two arrows starting at $i$, and there
are two arrows ending at $i$.
If $v=w$ then the edge $i$ occurs twice in the cyclic ordering 
of edges adjacent to $v$, so again there are two arrows starting at $i$ and
two arrows ending at $i$.

\smallskip

(1b) \ We define an (admissible) permutation $g$ on the arrows.
Given $\alpha: i\to j$, let $v$ be the vertex
such that $\alpha$ is attached to $v$,
then there is a unique edge $k$ adjacent to $v$ 
such that $i, j, k$ are consecutive edges in the ordering around $v$, 
and hence a unique  
arrow $\beta: j\to k$, also \lq attached\rq{}  to $v$, 
 and we set $g(\alpha):= \beta$.
This defines an admissible permutation on the arrows.
Writing $g$  as a product of disjoint cycles, gives  
a bijection between the cycles of $g$ and the vertices of $\Gamma$. 
Namely, let the cycle of $g$  correspond to $v$ 
if it consists of the arrows attached to $v$.

\smallskip

(2) \ Suppose we are given a connected 2-regular quiver $Q$ 
and an admissible permutation $g$, written as a product 
of disjoint cycles.
Define a graph $\Gamma$ with
 vertices  the cycles of $g$, and  edges  the vertices of $Q$.
Each cycle of $g$ 
defines a cyclic ordering of the edges adjacent to
 the vertex corresponding to  this cycle. Hence we get a Brauer graph.

\smallskip

(3)  It is clear that these give a bijection.
\end{proof}

\begin{remark}\label{rem:2.5} 
In part (1b) of the above proof, we may have $i=j$. There are two such cases. 
If the edge $i$ is adjacent to two distinct vertices of $\Gamma$
then  $i$ is the only edge adjacent to a vertex $v$ and 
we have $g(\alpha)=\alpha$. We call $\alpha$ an \emph{external} loop. 
Otherwise the edge 
$i$ is a loop of $\Gamma$, and then $g(\alpha)\neq \alpha$. 
In  this case the cycle of $g$ passes twice through vertex $i$
of the quiver. 
We call $\alpha$ an \emph{internal} loop.
\end{remark}

The Brauer graph $\Gamma$ comes with a multiplicity 
function $e$ defined on the vertices. 
Given $(Q, g)$, we take the same multiplicity function,
 defined on the cycles of $g$, which gives the  function $m_{\bullet}$
 which we  have called a weight function.
The permutation $g$ determines the permutation $f$ of the arrows where 
 $f(\alpha) = \overline{g(\alpha)}$ for any arrow $\alpha$.
 Clearly $f$ is also admissible, and $f$ and $g$ determine each other.

We have seen that the combinatorial data for $B_{\Gamma}$ are the same as
the combinatorial data for $B(Q, f, m_{\bullet})$. Therefore
$B_{\Gamma}$ is in fact equal to $B(Q, f, m_{\bullet})$.

In the definition of a biserial quiver we focus on $(Q, f)$, this is motivated
by the connection to 
biserial weighted surface algebras, which 
we will define later.

\bigskip

The following compares various algebras. 
The equivalence of the statements (i) and (iii)
was already obtained  by Roggenkamp in 
\cite[Sections 2 and 3]{Ro} 
(see also \cite[Proposition~1.2]{Ai} and \cite[Theorem~1.1]{Sch}).
We include it, for completeness.

\begin{theorem}
\label{th:2.6}
Let $A$ be a basic, indecomposable algebra of dimension 
at least $2$, over an algebraically closed field $K$.
The following  are equivalent:
\begin{enumerate}[(i)]
  \item
    $A$ is a Brauer graph algebra.
 \item
  $A$ is isomorphic to an algebra $B(Q,f,m_{\bullet})$ 
  where $(Q,f,m_{\bullet})$ is a (weighted)  biserial quiver.
 \item
  $A$ is a symmetric special biserial algebra.
\end{enumerate}
\end{theorem}

\begin{proof} As we have just seen,   
(i) and (ii) are equivalent.
The implication 
(ii) $\Rightarrow$ (iii)
follows from
Proposition \ref{prop:2.3}.

We prove now 
(iii) $\Rightarrow$ (ii).
Assume that $A$ is a basic symmetric special biserial algebra, let
$A= KQ_A/I$ where $Q_A$ is the Gabriel quiver of $A$. 
We will define a (weighted) biserial quiver $(Q,f,m_{\bullet})$
and show that $A$ is isomorphic to $B(Q,f,m_{\bullet})$.
Since $A$ is special biserial, for each vertex $i$ of $Q_A$,
we have $|s^{-1}(i)| \leq 2$ and $|t^{-1}(i)| \leq 2$.
The algebra $A$ is  symmetric, therefore for each vertex $i \in Q_0$,
we have $|s^{-1}(i)|= |t^{-1}(i)|$: Namely, if
$|s^{-1}(i)| = 1$ then by the special biserial relations, the
projective module $e_iA$ is uniserial. It is isomorphic
to the injective hull of the simple module $S_i$, and hence
$|t^{-1}(i)|=1$.  
If $|t^{-1}(i)|=1$ then by the same reasoning, applied to 
$D(Ae_i) \cong e_iA$ it follows that $|s^{-1}(i)| = 1$.

Let $\Delta:= \{ i\in (Q_A)_0 \mid   |s^{-1}(i)|= 1 \}$, to each $i\in \Delta$ we adjoin 
a loop $\eta_i$ at $i$ to the quiver $Q_A$, which then gives a 2-regular
quiver. Explicitly, let 
$Q := (Q_0,Q_1,s,t)$ with $Q_0 = (Q_A)_0$ and
$Q_1$ is the disjoint union $(Q_A)_1\bigcup \{ \eta_i: i\in \Delta\}.$ 

We define a permutation $f$ of $Q_1$.
For each $i \in \Delta$,
there are unique arrows $\alpha_i$ and $\beta_i$ in $Q_A$ with
$t(\alpha_i) = i = s(\beta_i)$,
and we set $f(\alpha_i) = \eta_i$ and $f(\eta_i) = \beta_i$.
If $\alpha$ is any  arrow of $Q_A$ 
with $t(\alpha)$ not in $\Delta$,
we define $f(\alpha)$ to be the unique arrow in $(Q_A)_1$
with $\alpha f(\alpha) \in I$.
With this, $(Q, f)$ is a biserial quiver.

We define now a weight function
$m_{\bullet} : \cO(g) \to \bN^*$,
where $g = \bar{f}$.
For each $j \in \Delta$,
we have $g(\eta_j)  = \eta_j$, 
and we set $m_{\cO(\eta_j)} = 1$.
Let $\alpha$ be some arrow of $Q_A$ 
starting at vertex $i$, and
let 
$n_{\alpha} = |\cO(\alpha)|$.
Since $A$ is  symmetric special biserial,
there exists $m_{\alpha} \in \bN^*$ such  that
\[
   B_{\alpha} := 
     \Big(\alpha g(\alpha) \dots g^{n_{\alpha}-1}(\alpha)\Big)^{m_{\alpha}}
\]
is a maximal cyclic path in $Q_A$
which does not belong to $I$,
and spans the socle 
of the indecomposable projective module
$e_i A$.
The integer $m_{\alpha}$ is constant on the
$g$-orbit of $\alpha$ and we may define 
$m_{\cO(\alpha)} = m_{\alpha}$.

It remains to show that by suitable scaling of arrows one
obtains the stated relations involving paths $B_{\alpha}$.
Fix a symmetrizing linear form $\vf$ for $A$.  
Fix an orbit of $g$, say $\cO(\nu)$, 
there is a non-zero scalar $d_{\nu}$ such
that for all arrows $\alpha$ in this orbit we have
$$\vf(B_{\alpha}) \ = \ d_{\nu}.
$$
We may assume $d_{\nu}=1$. 
Namely, 
we can choose in $\cO(\nu)$ an
arrow, $\alpha$ say, and replace it by $\lambda \alpha$ where
$\lambda^{m_{\alpha}} = d_{\nu}^{-1}$.  
The cycles are disjoint, and if we do this for each cycle 
then we have $\vf(B_{\alpha}) = 1$ for all arrows $\alpha$.

Let $i$ be a vertex of $Q_A$ with $|s^{-1}(i)| = 2$,
and  let $\alpha, \bar{\alpha}$ be the two arrows starting at $i$.
Then there are non-zero scalars 
$c_{\alpha}$ and $c_{\bar{\alpha}}$ 
such that  
$c_{\alpha} B_{\alpha} = c_{\bar{\alpha}} B_{\bar{\alpha}}$
in $A$.
Then we have 
$$c_{\alpha} = c_{\alpha}\vf(B_{\alpha}) 
= \vf(c_{\alpha}B_{\alpha}) =  \vf(c_{\ba}B_{\ba}) = c_{\ba}\vf(B_{\ba})
= c_{\ba}.
$$
Hence  we can  cancel these scalars and  obtain the required relations.
With this, there is a canonical isomorphism of $K$-algebras
$A = K Q_A / I \to B(Q,f,m_{\bullet})$.
\end{proof}

\medskip

We will from now 
suppress the word 'weighted', in analogy to the convention 
for Brauer graph algebras, where the multiplicity function 
is part of the definition but is not explicitly mentioned.

\bigskip
We will study idempotent algebras, and it is important
that any idempotent algebra of a special biserial symmetric algebra
is again special biserial symmetric.

\begin{proposition}
\label{prop:2.7}
Let $A$ be a symmetric special biserial algebra.
Assume $e$ is an idempotent of $A$ which is a sum of some of the
$e_i$ associated to vertices of $Q_A$.
Then $e A e$ also is a symmetric special biserial algebra.
\end{proposition} 

\begin{proof}
We may assume that 
$A = B(Q,f,m_{\bullet})$
for a weighted biserial quiver $(Q,f,m_{\bullet})$
and $e A e$ is indecomposable, and let 
$Q = (Q_0,Q_1,s,t)$.
We will show that
$eAe = (\tilde{Q},\tilde{f},\tilde{m}_{\bullet})
 = K Q/J(\tilde{Q},\tilde{f},\tilde{m}_{\bullet})$
for a weighted biserial quiver
$(\tilde{Q},\tilde{f},\tilde{m}_{\bullet})$.
We define $\tilde{Q}_0$ to be the set of all vertices $i \in Q_0$
such that $e$ is the sum of the primitive idempotents
$e_i$.
For each arrow $\alpha \in Q_1$ with
$s(\alpha) \in \tilde{Q}_0$, we denote by $\tilde{\alpha}$
the shortest path in $Q$ of the form
$\alpha g(\alpha) \dots g^p(\alpha)$ with
$p \in \{0,1,\dots,n_{\alpha}-1\}$ and
$t(g^p(\alpha)) \in \tilde{Q}_0$.
Such  a path exists because
$\alpha g(\alpha) \dots g^{n_{\alpha}-1}(\alpha)$
is a cycle around vertex  $s(\alpha) =t(g^{n_{\alpha}-1}(\alpha))$ in $\tilde{Q}_0$.
Then we define $\tilde{Q}_1$ to be set of paths $\tilde{\alpha}$ in $Q$
for all arrows $\alpha \in Q_1$ with $s(\alpha) \in \tilde{Q}_0$.
Moreover, for
$\tilde{\alpha} = \alpha g(\alpha) \dots g^p(\alpha)$,
we set
$\tilde{s} (\tilde{\alpha}) = s(\alpha)$
and
$\tilde{t} (\tilde{\alpha}) = t(g^p(\alpha))$.
This defines a $2$-regular quiver
$\tilde{Q} = (\tilde{Q}_0,\tilde{Q}_1,\tilde{s},\tilde{t})$.
Further, for each arrow
$\tilde{\alpha} = \alpha g(\alpha) \dots g^p(\alpha)$
in $\tilde{Q}_1$,
there is exactly one arrow
$\tilde{\beta} = \beta g(\beta) \dots g^r(\beta)$
in $\tilde{Q}_1$
such that
$\tilde{t} (\tilde{\alpha}) = t(g^p(\alpha)) 
 = s(\beta) = \tilde{s} (\tilde{\beta})$
and $f(\alpha) = \beta$,
and we set
$\tilde{f} (\tilde{\alpha}) = \tilde{\beta}$.
This defines a biserial quiver $(\tilde{Q},\tilde{f})$.
Let $\tilde{g}$ be the permutation of $\tilde{Q}_1$
associated to $\tilde{f}$, and $\cO(\tilde{g})$ the set
of $\tilde{g}$-orbits in $\tilde{Q}_1$.
Then we define the weight function
$\tilde{m}_{\bullet} : \cO(\tilde{g}) \to \bN^*$
of $(\tilde{Q},\tilde{f})$ by setting
$\tilde{m}_{\cO(\tilde{\alpha})} = m_{\cO(\alpha)}$
for each arrow $\tilde{\alpha} \in \tilde{Q}_1$.
With these,  the biserial quiver algebra
$B(\tilde{Q},\tilde{f},\tilde{m}_{\bullet})
 = K \tilde{Q}/J(\tilde{Q},\tilde{f},\tilde{m}_{\bullet})$
is isomorphic to $eAe$.
\end{proof}

We end this section with an example illustrating 
Theorem~\ref{th:2.6}.
This also 
shows that  an idempotent algebra of a Brauer graph algebra
need not be indecomposable, by taking   $e = 1_{B_{\Gamma}} - e_4$.

\begin{example}
\label{ex:2.8}
Let $\Gamma$ be the Brauer graph
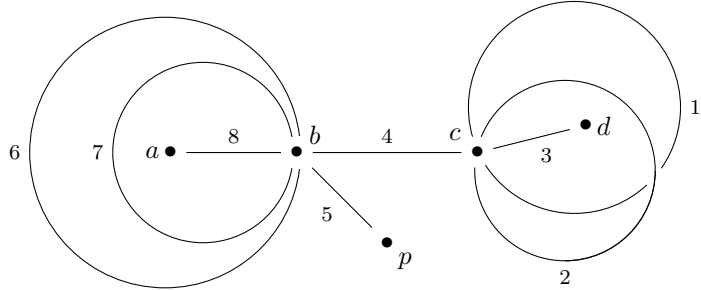
\[
\begin{tikzpicture}
[-,scale=1.2]
\coordinate (3) at (-3,0);
\coordinate (6) at (-1,0);

\node (a) at (-2.4,0) {$\bullet$};
\node at (-2.6,0) {$a$};
\node (b) at (-1,0) {$\bullet$};
\node at (-.8,0.2) {$b$};
\node (c) at (1,0) {$\bullet$};
\node at (.75,0.2) {$c$};
\node (p) at (0,-1) {$\bullet$};
\node at (.2,-1.2) {$p$};

\node (d) at (2.2,0.3) {$\bullet$};
\node at (2.4,0.3) {$d$};

\draw(-3.04,0) node[left]{\footnotesize$7$} arc (180:350:1);
\draw(-3.04,0) arc (180:10:1);
\draw(-3.96,0) node[left]{\footnotesize$6$} arc (180:353:1.5);
\draw(-3.96,0) arc (180:7:1.5);

\draw(2.97,-.2) arc (0:158:1);
\draw(2.97,-.2) arc (360:178:1);

\draw(2.97,-.2) arc (360:270:1) node[below]{\footnotesize$2$};

\draw(3.25,.5) node[right]{\footnotesize$1$} arc (0:196:1.175);
\draw(3.25,.5) arc (360:325:1.175);
\draw(2.075,-.675) arc (270:313:1.175);
\draw(2.075,-.675) arc (270:213:1.175);

\draw 
(a) edge node[above]{\footnotesize$8$} (b)
(b) edge node[above]{\footnotesize$4$} (c)
(b) edge node[below left]{\footnotesize$5$} (p)
(c) edge node[below right]{\footnotesize$3$} (d)
;
\end{tikzpicture}
\]
where we take the clockwise ordering
of the edges around each vertex.
Then $B_{\Gamma}$ is the symmetric algebra
$B(Q,f,m_{\bullet})$ with 
 biserial quiver $(Q,f)$ 
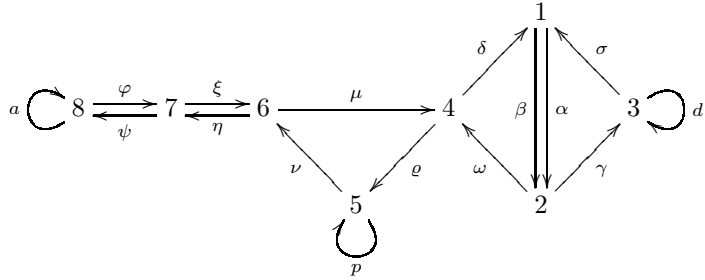
\[
\xymatrix{
	&&&&& 
	1
	\ar@<-.5ex>[dd]_{\beta}
	\ar@<+.5ex>[dd]^{\alpha}
	\\ 
	8
	\ar@<+.5ex>[r]^{\varphi}
	\ar@(dl,ul)^{a}[] 
	& 
	7
	\ar@<+.5ex>[l]^{\psi}
	\ar@<+.5ex>[r]^{\xi}
	& 
	6
	\ar@<+.5ex>[l]^{\eta}
	\ar[rr]^{\mu}
	&& 
	4
	\ar[ru]^{\delta}
	\ar[ld]^{\varrho}
	&&
	3
	\ar[lu]_{\sigma}
	\ar@(ur,dr)^{d}[]
	\\ 
	&&&
	5
	\ar[lu]^{\nu}
	\ar@(rd,ld)^{p}[]
	&&
	2
	\ar[lu]^{\omega}
	\ar[ru]_{\gamma}
} 
\]
where the  $f$-orbits are
$(\alpha \ \omega \ \varrho \ p \ \nu \ \mu \ \delta \ \beta \ \gamma \ d \ \sigma)$,
$(\eta \ \xi)$,
$(a \ \varphi \ \psi)$. Then
the $g$-orbits are
$\cO(a) = (a)$,
$\cO(d) = (d)$,
$\cO(p) = (p)$,
$$\cO(\alpha) = (\alpha \ \gamma \ \sigma \ \beta \ \omega \ \delta), \ \ 
\cO(\varrho) = (\varrho \ \nu \ \eta \ \psi \ \varphi \ \xi \ \mu).
$$
The weight function $m_{\bullet} : \cO(g) \to \bN^*$
is as before given by the multiplicity function of the Brauer graph 
$\Gamma$. We note that 
$C_{\alpha} = \alpha \gamma \sigma \beta \omega \delta$
and 
$C_{\bar{\alpha}} = C_{\beta} 
 = \beta \omega \delta  \alpha \gamma \sigma$,
and
$v(\alpha) = c = v(\beta)$. 
\end{example}

\section{Biserial weighted surface algebras}\label{sec:bisweight}

In this section 
we introduce  biserial weighted surface algebras
and describe their basic properties.

In this  paper, by a \emph{surface}
we mean a connected, compact, $2$-dimensional real
manifold $S$, orientable or non-orientable,
with boundary or without boundary.
It is well known that every surface $S$ admits
an additional structure of a finite
$2$-dimensional triangular cell complex,
and hence a triangulation (by the deep Triangulation Theorem
(see for example \cite[Section~2.3]{Ca})).

For a positive natural number $n$, we denote by $D^n$ the unit disk
in the $n$-dimensional Euclidean space $\bR^n$,
formed by all points of distance $\leq 1$ from the origin.
Then the boundary $\partial D^n$ of $D^n$ is the unit sphere
$S^{n-1}$ in $\bR^n$, formed by all points of distance $1$
from the origin.
Further, by an $n$-cell we mean a topological space
homeomorphic to the open disk 
$\intt D^n = D^n \setminus \partial D^n$.
In particular,
$S^0 = \partial D^1$ consists of two points.
Moreover, we define
$D^0 = \intt D^0$ to be a point.

We refer to \cite[Appendix]{H} for some basic topological
facts about cell complexes.

Let $S$ be a surface.
In the paper, by a 
\emph{finite $2$-dimensional triangular cell complex structure} 
on $S$ we mean a finite family of continuous maps
$\varphi_i^n : D_i^n \to S$, with $n \in \{0,1,2\}$
and $D_i^n = D^n$,
satisfying the following conditions:
\begin{enumerate}[(1)]
 \item
  Each $\varphi_i^n$ restricts to a homeomorphism from
  $\intt D_i^n$ to the $n$-cell $e_i^n = \varphi_i^n(\intt D_i^n)$ of $S$,
  and these cells are all disjoint and their union is $S$.
 \item
  For each $2$-dimensional cell $e_i^2$, $\varphi_i^2(\partial D_i^2)$
  is the union of $k$ $1$-cells and $m$ $0$-cells,
  with $k \in \{2,3\}$ and $m \in \{1,2,3\}$.
\end{enumerate}
Then the closures  $\varphi_i^2(D_i^2)$ of all $2$-cells $e_i^2$
are called \emph{triangles} of $S$,
and the closures  $\varphi_i^1(D_i^1)$ of all $1$-cells $e_i^1$
are called \emph{edges} of $S$.
The collection $T$ of all triangles  $\varphi_i^2(D_i^2)$ 
is said to be a \emph{triangulation} of $S$.
We assume that such a triangulation $T$ of $S$
has at least two  different edges,
or equivalently, there are at least two  different
$1$-cells in the considered triangular cell complex structure on $S$.
Then $T$ is a finite collection $T_1,\dots,T_n$ of triangles
of the form
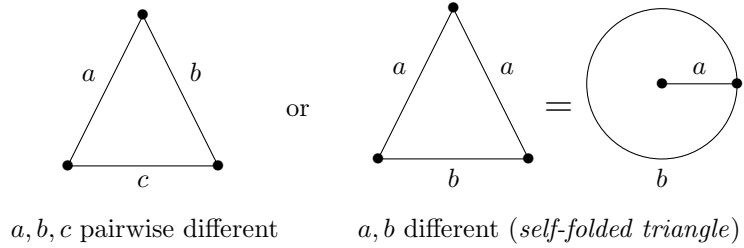
\begin{gather*}
\qquad
\begin{tikzpicture}[auto]
\coordinate (a) at (0,2);
\coordinate (b) at (-1,0);
\coordinate (c) at (1,0);
\draw (a) to node {$b$} (c)
(c) to node {$c$} (b);
\draw (b) to node {$a$} (a);
\node (a) at (0,2) {$\bullet$};
\node (b) at (-1,0) {$\bullet$};
\node (c) at (1,0) {$\bullet$};
\end{tikzpicture}
\qquad
\raisebox{7ex}{\mbox{or}}
\qquad
\begin{tikzpicture}[auto]
\coordinate (a) at (0,2);
\coordinate (b) at (-1,0);
\coordinate (c) at (1,0);
\draw (c) to node {$b$} (b)
(b) to node {$a$} (a);
\draw (a) to node {$a$} (c);
\node (a) at (0,2) {$\bullet$};
\node (b) at (-1,0) {$\bullet$};
\node (c) at (1,0) {$\bullet$};
\end{tikzpicture}
%
%
\raisebox{7ex}{\LARGE =}
\ \,
\begin{tikzpicture}[auto]
\coordinate (c) at (0,0);
\coordinate (a) at (1,0);
\coordinate (b) at (0,-1);
\draw (c) to node {$a$} (a);
\draw (b) arc (-90:270:1) node [below] {$b$};
\node (a) at (1,0) {$\bullet$};
\node (c) at (0,0) {$\bullet$};
\end{tikzpicture}
%
\\
\mbox{$a,b,c$ pairwise different}
\qquad
\quad
\mbox{$a,b$ different (\emph{self-folded triangle})}
\end{gather*}
such that every edge of such a triangle in $T$ is either
the edge of exactly two triangles, is the self-folded
edge, or lies on the boundary.
We note that a given surface $S$ admits many
finite $2$-dimensional triangular cell complex structures, 
and hence triangulations.
We refer to \cite{Ca,KC,Ki} for
general background on surfaces and
constructions of surfaces from plane models.

Let $S$ be a surface and
$T$ a triangulation $S$.
To each triangle $\Delta$ in $T$ we may associate an orientation
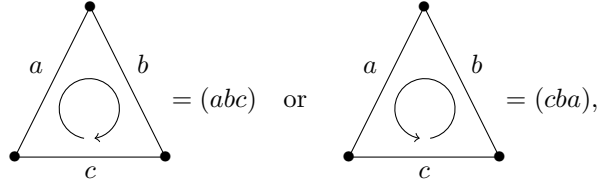
\[
\begin{tikzpicture}[auto]
\coordinate (a) at (0,2);
\coordinate (b) at (-1,0);
\coordinate (c) at (1,0);
\coordinate (d) at (-.08,.25);
\draw (a) to node {$b$} (c)
(c) to node {$c$} (b)
(b) to node {$a$} (a);
\draw[->] (d) arc (260:-80:.4);
\node (a) at (0,2) {$\bullet$};
\node (b) at (-1,0) {$\bullet$};
\node (c) at (1,0) {$\bullet$};
\end{tikzpicture}
\raisebox{7ex}{\!\!$=(abc)$}
\raisebox{7ex}{\quad or \ \ }
\begin{tikzpicture}[auto]
\coordinate (a) at (0,2);
\coordinate (b) at (-1,0);
\coordinate (c) at (1,0);
\coordinate (d) at (.08,.25);
\draw (a) to node {$b$} (c)
(c) to node {$c$} (b)
(b) to node {$a$} (a);
\draw[->] (d) arc (-80:260:.4);
\node (a) at (0,2) {$\bullet$};
\node (b) at (-1,0) {$\bullet$};
\node (c) at (1,0) {$\bullet$};
\end{tikzpicture}
\raisebox{7ex}{\!\!$=(cba)$,}
\]
if $\Delta$ has pairwise different edges $a,b,c$, and
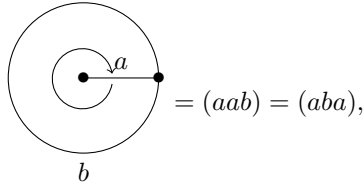
\[
\begin{tikzpicture}[auto]
\coordinate (c) at (0,0);
\coordinate (a) at (1,0);
\coordinate (b) at (0,-1);
\coordinate (d) at (.38,-.08);
\draw (c) to node {$a$} (a);
\draw (b) arc (-90:270:1) node [below] {$b$};
\node (a) at (1,0) {$\bullet$};
\node (c) at (0,0) {$\bullet$};
\draw[->] (d) arc (-10:-350:.4);
\end{tikzpicture}
\raisebox{7ex}{$=(aab)=(aba)$,}
\]
if $\Delta$ is self-folded, with the self-folded edge $a$,
and the other edge $b$.
Fix an orientation of each triangle $\Delta$ of $T$,
and denote this choice by $\vv{T}$.
Then
the pair $(S,\vv{T})$ is said to be a
\emph{directed triangulated surface}.
To each directed triangulated surface $(S,\vv{T})$
we associate the quiver $Q(S,\vv{T})$ whose vertices
are the edges of $T$ and the arrows are defined as
follows:
\begin{enumerate}[(1)]
 \item
  for any oriented triangle $\Delta = (a b c)$ in $\vv{T}$
  with pairwise different edges $a,b,c$, we have the cycle
  
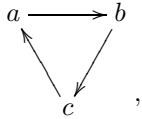
\[
    \xymatrix@C=.8pc{a \ar[rr] && b \ar[ld] \\ & c \ar[lu]}
    \raisebox{-7ex}{,}
  \]
 \item
  for any self-folded triangle $\Delta = (a a b)$ in $\vv{T}$,
  we have the quiver
  
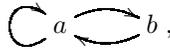
\[
    \xymatrix{ a \ar@(dl,ul)[] \ar@/^1.5ex/[r] & b \ar@/^1.5ex/[l]} ,
  \]
 \item
  for any boundary edge $a$ in ${T}$,
  we have the loop
  
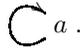
\[
    {\xymatrix{ a \ar@(dl,ul)[]}} .
  \]
\end{enumerate}
Then $Q = Q(S,\vv{T})$ is a triangulation quiver in
the following sense (introduced independently by Ladkani
in \cite{La}).

A \emph{triangulation quiver} is a pair $(Q,f)$,
where $Q = (Q_0,Q_1,s,t)$ is a finite connected quiver
and $f : Q_1 \to Q_1$ is a permutation
on the set $Q_1$ of arrows of $Q$ satisfying
the following conditions:
\begin{enumerate}[(a)]
 \item[(a)]
  every vertex $i \in Q_0$ is the source and target of exactly two
  arrows in $Q_1$,
 \item[(b)]
  for each arrow $\alpha \in Q_1$, we have $s(f(\alpha)) = t(\alpha)$,
 \item[(c)]
  $f^3$ is the identity on $Q_1$.
\end{enumerate}
Hence, a triangulation quiver $(Q,f)$ is a biserial quiver
$(Q,f)$ such that $f^3$ is the identity.

For the quiver $Q = Q(S,\vv{T})$ of a
directed triangulated surface $(S,\vv{T})$,
the permutation $f$ on its set of arrows
is defined as follows:
\begin{enumerate}[(1)]
 \item
  \raisebox{3ex}%
  {\xymatrix@C=.8pc{a \ar[rr]^{\alpha} && b \ar[ld]^{\beta} \\ & c \ar[lu]^{\gamma}}}%
  \quad
  $f(\alpha) = \beta$,
  $f(\beta) = \gamma$,
  $f(\gamma) = \alpha$,

  for an oriented triangle $\Delta = (a b c)$ in $\vv{T}$,
  with pairwise different edges $a,b,c$,
 \item
  \raisebox{0ex}%
  {\xymatrix{ a \ar@(dl,ul)[]^{\alpha} \ar@/^1.5ex/[r]^{\beta} & b \ar@/^1.5ex/[l]^{\gamma}}}
  \quad
  $f(\alpha) = \beta$,
  $f(\beta) = \gamma$,
  $f(\gamma) = \alpha$,

  for a self-folded triangle $\Delta = (a a b)$ in $\vv{T}$,\vspace{1mm}
 \item
  \raisebox{0ex}%
  {\xymatrix{ a \ar@(dl,ul)[]^{\alpha}}}
  \quad
  $f(\alpha) = \alpha$,\vspace{1mm}

  for a boundary edge $a$ of ${T}$.
\end{enumerate}

If $(Q, f)$ is a 
triangulation quiver, then the quiver
$Q$ is  $2$-regular.
\emph{We will  consider only the triangulation
quivers with at least two vertices.}
Note that  different directed triangulated
surfaces (even of different genus) may lead to the same
triangulation quiver (see \cite[Example~4.4]{ESk3}).


The following theorem 
is a slightly stronger version of 
\cite[Theorem~4.11]{ESk3}
(see also \cite[Example~8.2]{ESk4} for the case with two vertices).

\begin{theorem}
\label{th:3.1}
Let $(Q,f)$ be a triangulation quiver
with at least two vertices.
Then there exists a directed triangulated surface
$(S,\vv{T})$ such that 
$S$ is orientable, 
$\vv{T}$ is a coherent orientation of triangles in $T$, and
$(Q,f) = (Q(S,\vv{T}),f)$.
\end{theorem}

\begin{proof}
This is a minor adjustment of the proof of
Theorem~4.11 in \cite{ESk3} which we will now present.
We denote by $n(Q,f)$ 
the number of $f$-orbits in $Q_1$ of length $3$.
Note that $n(Q,f) \geq 1$ because $Q$ has at least two vertices.
There is exactly one triangulation quiver with two vertices, namely
\[
  \xymatrix{
     a \ar@(dl,ul)[]^{\alpha} \ar@/^1.5ex/[r]^{\beta} 
     & b \ar@/^1.5ex/[l]^{\gamma} \ar@(ur,dr)[]^{\sigma}
  }
\]
with
$f(\alpha) = \beta$,
$f(\beta) = \gamma$,
$f(\gamma) = \alpha$,
$f(\sigma) = \sigma$,
and it is the triangulation quiver associated to the self-folded
triangulation of the disk
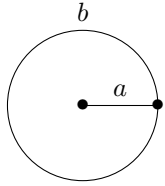
\[
\begin{tikzpicture}[auto]
\coordinate (c) at (0,0);
\coordinate (a) at (1,0);
\coordinate (b) at (0,1);
\draw (c) to node {$a$} (a);
\draw (b) arc (90:-270:1) node [above] {$b$};
\node (a) at (1,0) {$\bullet$};
\node (c) at (0,0) {$\bullet$};
\end{tikzpicture}
\]
with $b$ a boundary edge.
It is also known that the theorem holds for all triangulation quivers
with three vertices (see  \cite[Examples 4.3 and 4.4]{ESk4} and 
Example~\ref{ex:4.5}). 
Therefore, we may assume that  $n(Q,f) \geq 2$ and 
$Q$ has at least four vertices.
Now our induction assumption is:
For any triangulation quiver $(Q',f')$ with at least two vertices
and $n(Q',f') < n(Q,f)$ 
there exists a directed triangulated surface 
$(S',\vv{T'})$ such that $S'$ is orientable,
$\vv{T'}$ is a coherent orientation of triangles in $T'$,
and $(Q',f') = (Q(S',\vv{T'}),f')$.
Then we proceed as in the reconstruction steps (1) and (2)
of the proof of \cite[Theorem~4.11]{ESk3},
with the following adjustments.
In  step (1), we replace the projective plane $\mathbb{P}$
by the disk with self-folded triangulation,
described above.
In  step (2), we glue the oriented triangle
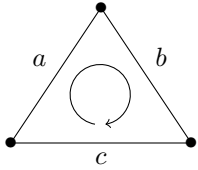
\[
\begin{tikzpicture}[auto]
\coordinate (a) at (0,1.8);
\coordinate (b) at (-1.2,0);
\coordinate (c) at (1.2,0);
\coordinate (d) at (-.08,.25);
\draw (a) to node {$b$} (c)
(c) to node {$c$} (b)
(b) to node {$a$} (a);
\draw[->] (d) arc (260:-80:.4);
\node (a) at (0,1.8) {$\bullet$};
\node (b) at (-1.2,0) {$\bullet$};
\node (c) at (1.2,0) {$\bullet$};
\end{tikzpicture}
\]
with pairwise different edges, in a coherent way
with the corresponding triangles of the directed
triangulated surface $(S',\vv{T'})$,
constructed in this step.
\end{proof}

\begin{remark}
\label{rem:3.2}
There is an alternative proof of Theorem~\ref{th:3.1}.
According to Lemma~\ref{lem:2.4} and Theorem~\ref{th:2.6},
we may associate to a triangulation quiver $(Q,f)$
a Brauer graph $\Gamma$ with trivial multiplicity
function such that $B_{\Gamma} \cong B(Q,f,\mathds{1})$,
where $\mathds{1}$ is the trivial weight function of $(Q,f)$.
In the Brauer graph $\Gamma$, the vertices correspond
to the $g$-orbits in $Q_1$ and the edges to the
vertices of $Q$. Thickening the edges of $\Gamma$
we obtain an oriented surface whose border is given
by the faces of $\Gamma$,
corresponding to the $f$-orbits in $Q_1$.
Since $(Q,f)$ is a triangulation quiver, the faces are either triangles
or (internal) loops.
Capping now all triangle faces by disks $D^2$ we obtain 
a directed triangulated surface $(S,\vv{T})$
such that $(Q,f) = (Q(\cT,\vv{T}),f)$.
\end{remark}

\begin{remark}
\label{rem:3.3}
We would like to stress that the setting of directed
triangulated surfaces is natural for the purposes
of a self-contained representation theory 
of symmetric tame algebras of non-polynomial growth
which we are currently developing.
In particular, this gives the option of changing
orientation of any triangle independently,
keeping the same surface and triangulation.
\end{remark}

Let $(Q,f)$ be a triangulation quiver, this is in particular a
biserial quiver as introduced in Definition \ref{definition:2.2}. 
With the same notation, for a weight function
a weight function $m_{\bullet} : \cO(g) \to \bN^*$,
the associated weighted biserial quiver algebra
\[
  B(Q,f,m_{\bullet})
   = K Q / J(Q,f,m_{\bullet})
\]
is said to be a \emph{biserial weighted triangulation algebra}.
Moreover, if $(Q,f) = (Q(S, \vv{T}),f)$ 
for a directed triangulation surface   
$(S,\vv{T})$,
then $B(Q(S, \vv{T}),f,m_{\bullet})$
is called a \emph{biserial weighted surface algebra},
and denoted by $B(S, \vv{T}, m_{\bullet})$
(see  \cite{ESk3} and \cite{ESk4}).

 Biserial weighted surface algebras belong to the
class of algebras of generalized dihedral type,  
which generalize blocks of group algebras
with dihedral defect groups.  
They are introduced and studied in \cite{ESk4}.
We end this section by giving two examples of
biserial weighted surface algebras.

\begin{example}
\label{ex:3.3}
Consider the disk $D = D^2$ with the triangulation $T$
and orientation $\vv{T}$ of triangles in $T$ as follows
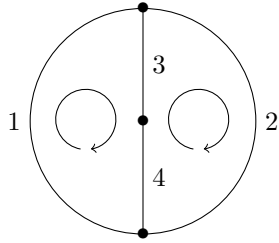
\[
\begin{tikzpicture}[auto]
\coordinate (o) at (0,0);
\coordinate (a) at (0,1.5);
\coordinate (b) at (-1.5,0);
\coordinate (c) at (1.5,0);
\coordinate (d) at (0,-1.5);
\coordinate (e) at (.67,-.37);
\coordinate (f) at (-.83,-.37);
\draw (a) to node {3} (o) to node {4} (d);
\draw[->] (e) arc (260:-80:.4);
\draw[->] (f) arc (260:-80:.4);
\draw (c) arc (0:180:1.5) node [left] {$1$} 
                  arc (180:360:1.5) node [right] {$2$} ;
\node (o) at (o) {$\bullet$};
\node (a) at (a) {$\bullet$};
\node (d) at (d) {$\bullet$};
\end{tikzpicture}
\]

Then the associated triangulation quiver
$(Q(D,\vv{T}),f)$ is of the form
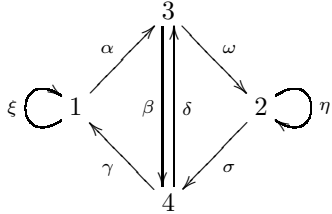
\[
  \xymatrix{
  & 
    3
      \ar@<-.5ex>[dd]_{\beta}
      \ar[rd]^{\omega}
  \\ 
    1
      \ar[ru]^{\alpha}
      \ar@(dl,ul)^{\xi}[]
  &&
    2
      \ar[ld]^{\sigma}
      \ar@(ur,dr)^{\eta}[]
  \\ 
  &
    4
      \ar[lu]^{\gamma}
      \ar@<-.5ex>[uu]_{\delta}
  } 
\]
with  $f$-orbits
$(\alpha\, \beta\, \gamma)$, 
$(\sigma\, \delta\, \omega)$, 
$(\xi)$, 
$(\eta)$.
Then  the $g$-orbits are
$\cO(\alpha) = (\alpha\, \omega\, \eta\, \sigma\, \gamma\, \xi)$ 
and
$\cO(\beta) = (\beta\, \delta)$. 
Hence a weight function
$m_{\bullet} : \cO(g) \to \bN^*$
is given by two positive integers 
$m_{\cO(\alpha)} = m$
and
$m_{\cO(\beta)} = n$.
Then the associated biserial weighted surface algebra
$B(D,\vv{T},m_{\bullet})$
is given by the above quiver and the relations:
\begin{align*}
 (\alpha  \omega \eta \sigma \gamma \xi)^m
  &= 
 (\xi \alpha  \omega \eta \sigma \gamma)^m
  ,
&
 \xi^2 &= 0, 
&
\alpha \beta &= 0,
&
\sigma \delta &= 0,
\\ 
 (\sigma \gamma \xi \alpha  \omega \eta)^m
  &= 
 (\eta \sigma \gamma \xi \alpha  \omega)^m
  ,
&
 \eta^2 &= 0, 
&
\beta \gamma &= 0,
&
\delta \omega &= 0,
\\
 (\omega \eta \sigma \gamma \xi \alpha )^m
  &= 
 (\beta \delta)^n
  ,
&
\!\!\!\!\!\!\!\!\!\!\!\!\!\!\!\!\!\!\!\!
 (\gamma \xi \alpha  \omega \eta \sigma)^m
  &= 
 (\delta \beta)^n
  ,
&
\gamma  \alpha &= 0,
&
\omega \sigma &= 0.
\end{align*}
\end{example}

\begin{example}
\label{ex:3.5}
Consider the torus $\bT$ with the triangulation $T$
and orientation $\vv{T}$ of triangles in $T$ as follows
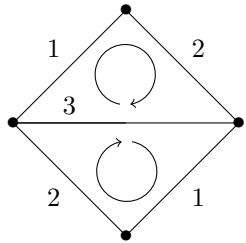
\[
\begin{tikzpicture}[auto]
\coordinate (o) at (0,0);
\coordinate (a) at (0,1.5);
\coordinate (b) at (-1.5,0);
\coordinate (c) at (1.5,0);
\coordinate (d) at (0,-1.5);
\coordinate (e) at (-.08,.25);
\coordinate (f) at (.08,-.25);
\draw (a) to node {2} (c)
(c) to node {1} (d)
(d) to node {2} (b)
(b) to node {1} (a);
\draw (b) to (c);
\draw (b) to node {3} (o);
\draw[->] (e) arc (260:-80:.4);
\draw[->] (f) arc (80:-260:.4);
\node (a) at (a) {$\bullet$};
\node (b) at (b) {$\bullet$};
\node (c) at (c) {$\bullet$};
\node (d) at (d) {$\bullet$};
\end{tikzpicture}
\]
Then the associated triangulation quiver
$(Q(\bT,\vv{T}),f)$ is of the form
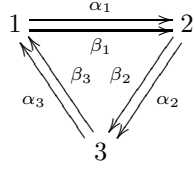
\[
  \xymatrix@R=3.pc@C=1.8pc{
    1
    \ar@<.45ex>[rr]^{\alpha_1}
    \ar@<-.45ex>[rr]_{\beta_1}
    && 2
    \ar@<.45ex>[ld]^{\alpha_2}
    \ar@<-.45ex>[ld]_{\beta_2}
    \\
    & 3
    \ar@<.45ex>[lu]^{\alpha_3}
    \ar@<-.45ex>[lu]_{\beta_3}
  }
\]
with  $f$-orbits
$(\alpha_1\, \alpha_2\, \alpha_3)$ and $(\beta_1\, \beta_2\, \beta_3)$.
\ 
Then $g$ has only one orbit which is 
$(\alpha_1\, \beta_2\, \alpha_3\, \beta_1\, \alpha_2\, \beta_3)$,
and hence
a weight function
$m_{\bullet} : \cO(g) \to \bN^*$
is given by a positive integer $m$.
Then the associated biserial weighted surface algebra
$B(\bT,\vv{T},m_{\bullet})$
is given by the above quiver and the relations:
\begin{align*}
 \alpha_1 \alpha_2 &= 0,
 &
 \beta_1 \beta_2 &= 0,
 &
  (\alpha_1 \beta_2 \alpha_3 \beta_1 \alpha_2 \beta_3)^m
    &=
  (\beta_1 \alpha_2 \beta_3 \alpha_1 \beta_2 \alpha_3)^m
,\\  
 \alpha_2 \alpha_3 &= 0,
 &
 \beta_2 \beta_3 &= 0,
 &
  (\alpha_2 \beta_3 \alpha_1 \beta_2 \alpha_3 \beta_1)^m
    &=
  (\beta_2 \alpha_3 \beta_1 \alpha_2 \beta_3 \alpha_1)^m
,\\  
 \alpha_3 \alpha_1 &= 0,
 &
 \beta_3 \beta_1 &= 0, 
 &
  (\alpha_3 \beta_1 \alpha_2 \beta_3 \alpha_1 \beta_2)^m
    &=
  (\beta_3 \alpha_1 \beta_2 \alpha_3 \beta_1 \alpha_2)^m
.
\end{align*}
The triangulation quiver $(Q(\bT,\vv{T}),f)$ is called 
the \lq Markov quiver\rq{}  
(see \cite{ESk4} for a motivation). 
\end{example}

\section{Proof of Theorem \ref{th:main1}}\label{sec:proof}

To prove the implication (ii) $\Rightarrow$ \ (i), let 
$B$ be a biserial weighted surface algebra. Then by 
 Theorem~\ref{th:3.1} we may assume $B= B(Q, f, m_{\bullet})$ where
$(Q,f)$ is a biserial quiver and $f^3$ is the identity. Then in particular
$B$ is a biserial quiver algebra, and by Theorem \ref{th:2.6}, 
we see that $B$ is a Brauer graph algebra. Now it follows from 
Theorem~\ref{th:2.6}
and Proposition~\ref{prop:2.7}
 that also $eBe$ is a Brauer graph algebra,  and 
(i) holds.

\medskip
We consider the implication (i)$\Rightarrow$ (ii). Assume $A$ is a Brauer
graph algebra, by Theorem \ref{th:2.6} we may assume $A=B(Q, f, m_{\bullet})$ where $(Q, f)$ is a biserial quiver. 
To obtain (ii), we  must find a biserial quiver $(Q^*, f^*)$ with
$(f^*)^3 = 1$ such that $A = e^*B^*e^*$ where $B^* = B(Q^*, f^*, m_{\bullet}^*)$
and $e^*$ an idempotent of $B^*$. 

The following shows that this can be done
in a canonical way, the construction gives an algorithm. Furthermore, applying
the construction twice gives an interesting consequence.

\begin{theorem}
\label{th:4.1} 
Let $B= B(Q, f, m_{\bullet})$ be a 
biserial quiver algebra.
Then there is a  canonically defined
weighted triangulation quiver $(Q^*,f^*,m_{\bullet}^*)$ 
such that the following statements hold.
\begin{enumerate}[(i)]
  \item
    $B$ is isomorphic to the idempotent algebra $e^* B^* e^*$
    of the biserial  triangulation algebra
    $B^* = B(Q^*,f^*,m_{\bullet}^*)$
    with respect to a canonically defined idempotent $e^*$ of $B^*$.    
  \item
    The triangulation quiver $(Q^*, f^*)$
    has no loops fixed by $f^*$.
  \item
    The triangulation quiver $(Q^{**}, f^{**})$
    has no loops and self-folded triangles.
  \item
    $B$ is isomorphic to the idempotent algebra $e^{**} B^{**} e^{**}$
    of the biserial  triangulation algebra
    $B^{**} = B(Q^{**},f^{**},m_{\bullet}^{**})$
    with respect to a canonically defined idempotent $e^{**}$ of $B^{**}$.    
\end{enumerate}
\end{theorem}

\begin{proof} 
Let  $Q = (Q_0,Q_1,s,t)$,
and let $g$ be the permutation of $Q_1$  associated to $f$.
We define a triangulation quiver $(Q^*, f^*)$ as follows.
We take $Q^* = (Q^*_0,Q^*_1,s^*,t^*)$ with
\begin{align*}
  Q_0^* & := Q_0 \cup \{x_{\alpha}\}_{\alpha \in Q_1} , &
  Q_1^* & := \{\alpha', \alpha'', \varepsilon_{\alpha} \}_{\alpha \in Q_1} 
\end{align*}
and 
$s^*(\alpha') = s(\alpha)$,
$t^*(\alpha') = x_{\alpha}$,
$s^*(\alpha'') = x_{\alpha}$,
$t^*(\alpha'') = t(\alpha)$,
$s^*(\varepsilon_{\alpha}) = x_{f(\alpha)}$,
$t^*(\varepsilon_{\alpha}) = x_{\alpha}$.
Moreover, we set 
$f^*(\alpha'') = f(\alpha)'$,
$f^*(f(\alpha)') = \varepsilon_{\alpha}$,
$f^*(\varepsilon_{\alpha}) = \alpha''$.
We observe that $(Q^*, f^*)$ is 
a triangulation quiver.
Let $g^*$ be the permutation of $Q_1^*$
associated to $f^*$.
We notice that,
for any arrow $\alpha$ of $Q$,
we have
$g^*(\alpha') = \alpha''$,
$g^*(\alpha'') = g(\alpha)'$,
and
$g^*(\varepsilon_{\alpha}) = \varepsilon_{f^{-1}(\alpha)}$.
For each arrow $\beta \in Q_1^*$,
we denote by $\cO^*(\beta)$ the $g^*$-orbit
of $\beta$.
Then the  $g^*$-orbits in $Q_1^*$ are 
\begin{align*}
  \cO^*(\alpha') & = \big(\alpha' \ \alpha'' \ g(\alpha)' \ g(\alpha)'' \ \dots 
                                \ g^{n_{\alpha}-1}(\alpha)' \ g^{n_{\alpha}-1}(\alpha)''\big) , \\
  \cO^*(\varepsilon_{\alpha}) & = \big(\varepsilon_{f^{r_{\alpha}-1}(\alpha)} 
                                \ \varepsilon_{f^{r_{\alpha}-2}(\alpha)} \ \dots 
                                \ \varepsilon_{f(\alpha)} \ \varepsilon_{\alpha}\big) , 
\end{align*}
for $\alpha \in Q_1$, where
$n_{\alpha}$ is the length of the $g$-orbit of $\alpha$
and
$r_{\alpha}$ is the length of the $f$-orbit of $\alpha$ 
in $Q_1$.
We define the weight function 
${m}_{\bullet}^*$
by 
${m}_{\cO^*(\alpha')}^* = m_{\alpha}$ 
and
${m}_{\cO^*(\varepsilon_{\alpha})}^* = 1$
for all $\alpha \in Q_1$. 

Let $B^* = B(Q^*,f^*,m_{\bullet}^*)$
be the biserial triangulation algebra
associated to $(Q^*,f^*,m_{\bullet}^*)$
and  let $e^*$ be the sum of the 
primitive idempotents
$e_i^*$ in $B^*$ associated to all vertices $i \in Q_0$.   
Using the proof of  Proposition~\ref{prop:2.7} 
we see directly that the idempotent algebra $e^*B^*e^*$ is isomorphic
to $B$. 
It follows also from the definition of $f^*$ that $Q^*$
has no loops fixed by $f^*$, and (ii) holds.
In particular, we conclude that 
$f^*(\varepsilon_{\alpha}) \neq \varepsilon_{\alpha}$
for any arrow $\alpha \in Q_1$.
Hence, the triangulation quiver $(Q^{**}, f^{**})$ has no loops,
and consequently it has also no self-folded triangles, and (iii) follows.
Finally, by (i), $B^*$ is isomorphic to an idempotent algebra 
$\hat{e}B^{**}\hat{e}$ of $B^{**} = B(Q^{**},f^{**},m_{\bullet}^{**})$
for the corresponding idempotent $\hat{e}$ of $B^{**}$.    
Taking $e^{**} = e^*\hat{e}$, we obtain that 
$B$ is isomorphic to the idempotent algebra $e^{**} B^{**} e^{**}$,
and hence (iv) also holds.
\end{proof} 

We give some illustrations for the $*$-construction.

\smallskip

(1) 
A  loop $\alpha$ in $Q$ fixed by $f$ is replaced
in $Q^*$ by the subquiver
\[
  \xymatrix{ 
     x_{\alpha} \ar@(dl,ul)[]^{\varepsilon_{\alpha}} \ar@/^1.5ex/[r]^{\alpha''}
       & s(\alpha) \ar@/^1.5ex/[l]^{\alpha'}
  }
\]
with the  $f^*$-orbit
$(\alpha' \ \varepsilon_{\alpha} \ \alpha'')$.

\smallskip

(2) 
A subquiver of $Q$ of the form
\[
  \xymatrix{ 
     a \ar@(dl,ul)[]^{\alpha} \ar@/^1.5ex/[r]^{\beta} & b \ar@/^1.5ex/[l]^{\gamma}
  }
\]
where
$(\alpha \ \beta \ \gamma)$ is an $f$-orbit,
is replaced in $Q^*$ by the quiver
\[
  \xymatrix{ 
     && x_{\beta} \ar[rd]^{\beta''} \ar@/_4ex/[lld]_{\varepsilon_{\alpha}} \\
     \!\! x_{\alpha} \ar@/^1.5ex/[r]^{\alpha''}  \ar@/_4ex/[rrd]_{\varepsilon_{\gamma}}
       & a \ar@/^1.5ex/[l]^{\alpha'} \ar[ru]^{\beta'} && b \ar[ld]^{\gamma'} \\
     && x_{\gamma} \ar[lu]^{\gamma''} \ar[uu]^{\varepsilon_{\beta}}      
  }
\]
with $f^*$-orbits
$(\alpha'' \ \beta' \ \varepsilon_{\alpha})$,
$(\gamma'' \ \alpha' \ \varepsilon_{\gamma})$ and $(\beta'' \ \gamma' \ \varepsilon_{\beta})$.

\smallskip

(3) 
A  subquiver of $Q$ of the form
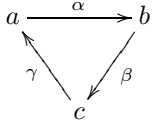
\[
  \xymatrix@C=1.2pc{ 
     a \ar[rr]^{\alpha} && b \ar[ld]^{\beta} \\ & c \ar[lu]^{\gamma}
  }
\]
and where
$(\alpha \ \beta \ \gamma)$ is an $f$-orbit,
is replaced in $Q^*$ by the quiver of the form
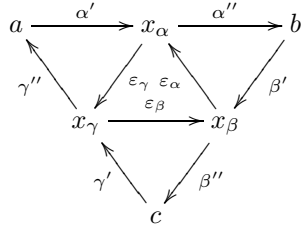
\[
  \xymatrix@C=1.2pc@!=1pc{ 
     a \ar[rr]^{\alpha'}
         && x_{\alpha} \ar[rr]^{\alpha''} \ar[ld]^{\varepsilon_{\gamma}}
         && b \ar[ld]^{\beta'} \\
     & x_{\gamma} \ar[rr]^{\varepsilon_{\beta}} \ar[lu]^{\gamma''} 
          && x_{\beta} \ar[ld]^{\beta''} \ar[lu]^{\varepsilon_{\alpha}} \\ 
     && c \ar[lu]^{\gamma'}
  }
\]
with $f^*$-orbits
$(\alpha'' \ \beta' \ \varepsilon_{\alpha})$,
$(\beta'' \ \gamma' \ \varepsilon_{\beta})$,
and $(\gamma'' \ \alpha' \ \varepsilon_{\gamma})$.

\begin{remark}
The statement (i) of the above theorem also holds 
if we replace the canonically defined weight function $m_{\bullet}^*$
by a weight function
$\bar{m}_{\bullet}^*$ 
such that 
$\bar{m}_{\cO^*(\alpha')} = m_{\alpha}$ 
and
$\bar{m}_{\cO^*(\varepsilon_{\alpha})}$
is an arbitrary positive integer, 
for any arrow $\alpha \in Q_1$. 
\end{remark}

\begin{remark}
\label{rem:4.3}
The construction of 
the triangulation quiver $(Q^*, f^*)$ 
associated to $(Q,f)$ is canonical, though a quiver with
fewer vertices may often be  sufficient.
In fact, it would be enough to apply the construction  only to  the arrows
in $f$-orbits  of length different from $1$ and $3$.
An
 algebra 
$B(Q,f,m_{\bullet})$
may have  many presentations as an idempotent algebra
of some biserial  triangulation algebra,
even for a triangulation quiver $(Q', f')$
with fewer   $f'$-orbits  than 
the number of $f^*$-orbits in the
triangulation quiver $(Q^*, f^*)$
(see Example~\ref{ex:4.7}).
\end{remark}

\begin{remark}
\label{rem:4.4}
The $*$-construction described in Theorem~\ref{th:4.1} 
provides a special class of triangulation quivers.
Namely, let $(Q,f)$ be a biserial quiver,
$g$ the permutation of $Q_1$ associated to $(Q,f)$,
and
$g^*$ the permutation of $Q_1^*$ associated to $(Q^*,f^*)$.
Then, for every arrow $\alpha \in Q_1$, we have in $Q_1^*$
the $g^*$-orbit $\cO^*(\alpha')$ of even length
$2|\cO(\alpha)|$ and
the $g^*$-orbit $\cO^*(\varepsilon_{\alpha})$ whose length
is the length of the $f$-orbit of $\alpha$ in $Q_1$.
In particular, all triangulation quivers $(Q',f')$
having only $g'$-orbits of odd length do not belong 
to this class of triangulation quivers.
For example, it is the case for the tetrahedral 
quiver considered in Section~\ref{sec:proof4}.
We refer also to \cite[Example~4.9]{ESk3} for an
example of triangulation quiver $(Q'',f'')$ for
which all arrows in $Q_1''$ belong to one 
$g''$-orbit of length $18$.
\end{remark}

\begin{example}
\label{ex:4.4}
Let $\Gamma$ be the Brauer tree
\[
 \xymatrix@R=3pc@C=3pc{
        \bullet \ar@{-}[r]^{1}
        \save[] +<-3mm,0mm> *{a} \restore 
        &
        \bullet
        \save[] +<3mm,0mm> *{b} \restore \\
  }
\]
with  multiplicity function 
$e(a) = m$ and $e(b) = n$.
Then the associated Brauer graph  algebra $B_{\Gamma}$
is the algebra
$B(Q,f,m_{\bullet})$
associated to the  biserial quiver 
$(Q,f,m_{\bullet})$
where $Q$ is of the form
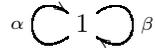
\[
 \xymatrix{
  1 \ar@(dl,ul)[]^{\alpha} \ar@(ur,dr)[]^{\beta}
 }
\]
with
$f(\alpha) = \beta$,
$f(\beta) = \alpha$,
$g(\alpha) = \alpha$,
$g(\beta) = \beta$,
and
$m_{\cO(\alpha)} = m$, and $m_{\cO(\beta)} = n$.
If  $m = 1$, then
$B_{\Gamma}$ is the truncated polynomial algebra
$K[x]/(x^{n+1})$.
The associated triangulation quiver $(Q^*,f^*)$
is of the form
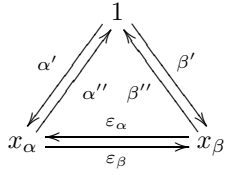
\[
  \xymatrix@R=3.pc@C=1.8pc{
    &1
    \ar@<-.45ex>[ld]_{\alpha'}
    \ar@<.45ex>[rd]^{\beta'}
    \\
    x_{\alpha} 
    \ar@<-.45ex>[ru]_{\alpha''}
    \ar@<-.45ex>[rr]_{\varepsilon_{\beta}}
    && x_{\beta} 
    \ar@<-.45ex>[ll]_{\varepsilon_{\alpha}}
    \ar@<.45ex>[lu]^{\beta''}
  }
\]
and the $f^*$-orbits are
$(\alpha'' \ \beta' \ \varepsilon_{\alpha})$
and
$(\beta'' \ \alpha' \ \varepsilon_{\beta})$.
Further,   the $g^*$-orbits are
$\cO^*(\alpha') = (\alpha' \ \alpha'' )$,
$\cO^*(\beta') = (\beta' \ \beta'')$,
$\cO^*(\varepsilon_{\alpha}) = (\varepsilon_{\alpha}\ \varepsilon_{\beta})$
and the weight function is
${m}^*_{\cO^*(\alpha')} = m$, 
${m}^*_{\cO^*(\beta')} = n$, 
${m}^*_{\cO^*(\varepsilon_{\alpha})} = 1$.
We also note that 
$(Q^*, f^*)$ is the triangulation quiver 
$(Q(\bT,\vv{T}), f)$ associated to the torus $\bT$
with  triangulation $T$ and orientation $\vv{T}$
of triangles in $T$ as follows
\[
\begin{tikzpicture}[auto]
\coordinate (o) at (0,0);
\coordinate (a) at (0,1.5);
\coordinate (b) at (-1.5,0);
\coordinate (c) at (1.5,0);
\coordinate (d) at (0,-1.5);
\coordinate (e) at (-.08,.25);
\coordinate (f) at (.08,-.25);
\coordinate (f2) at (-.08,-.25);
\draw (a) to node {2} (c)
(c) to node {1} (d)
(d) to node {2} (b)
(b) to node {1} (a);
\draw (b) to (c);
\draw (b) to node {3} (o);
\draw[->] (e) arc (260:-80:.4);
\draw[->] (f2) arc (-260:80:.4);
\node (a) at (a) {$\bullet$};
\node (b) at (b) {$\bullet$};
\node (c) at (c) {$\bullet$};
\node (d) at (d) {$\bullet$};
\end{tikzpicture}
\]
(compare with  Example~\ref{ex:3.5}).
\end{example}

\begin{example}
\label{ex:4.5}
Let $\Gamma$ be the Brauer graph
\[
\begin{tikzpicture}
[-,scale=.7]

\node (a) at (0,0) {$\bullet$};

\draw(0,2.04) node[above]{\footnotesize$1$} arc (90:260:1);
\draw(0,2.04) arc (90:-80:1);
\draw(a) node[below]{$a$};
\end{tikzpicture}
\]
with multiplicity
$e(a) = m$ for some $m \in \bN^*$.
Then the associated Brauer graph  algebra $B_{\Gamma}$
is the  algebra
$B(Q,f,m_{\bullet})$
where the quiver $Q$ is of the form
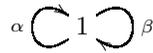
\[
 \xymatrix{
  1 \ar@(dl,ul)[]^{\alpha} \ar@(ur,dr)[]^{\beta}
 }
\]
with
$f(\alpha) = \alpha$,
$f(\beta) = \beta$,
$g(\alpha) = \beta$,
$g(\beta) = \alpha$,
and
$m_{\cO(\alpha)} = m$.
The associated triangulation quiver $(Q^*,f^*)$
is 
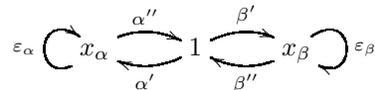
\[
  \xymatrix{ 
     x_{\alpha} \ar@(dl,ul)[]^{\varepsilon_{\alpha}} \ar@/^1.5ex/[r]^{\alpha''}
       & 1 \ar@/^1.5ex/[l]^{\alpha'} \ar@/^1.5ex/[r]^{\beta'}
       & \ar@/^1.5ex/[l]^{\beta''}
       x_{\beta} \ar@(ur,dr)[]^{\varepsilon_{\beta}} 
  }
\]
with $f^*$-orbits
$(\alpha'' \ \alpha' \ \varepsilon_{\alpha})$
and
$(\beta'' \ \beta' \ \varepsilon_{\beta})$.
Further,  the $g^*$-orbits are
$\cO^*(\alpha') = (\alpha' \ \alpha'' \ \beta' \ \beta'')$,
$\cO^*(\varepsilon_{\alpha}) = (\varepsilon_{\alpha})$,
$\cO^*(\varepsilon_{\beta}) = (\varepsilon_{\beta})$,
and
${m}^*_{\cO^*(\alpha')} = m$, 
${m}^*_{\cO^*(\varepsilon_{\alpha})} = 1$,
${m}^*_{\cO^*(\varepsilon_{\beta})} = 1$.
Note that 
$(Q^*, f^*)$ is the triangulation quiver 
$(Q(\bS,\vv{T}), f)$ associated to the sphere $\bS$
with  triangulation $T$ 
given by two self-folded triangles
\[
\begin{tikzpicture}[auto]
\coordinate (c) at (0,0);
\coordinate (a) at (1,0);
\coordinate (b) at (-1,0);
\coordinate (d) at (2,0);
\draw (c) to node {$2$} (a);
\draw (a) to node {$3$} (d);
\draw (b) arc (-180:180:1) node [left] {$1$};
\node (a) at (1,0) {$\bullet$};
\node (c) at (0,0) {$\bullet$};
\node (c) at (2,0) {$\bullet$};
\end{tikzpicture}
\]
where $\vv{T}$ is canonically defined.
\end{example}

\begin{example}
\label{ex:4.7}
Let $(Q,f,m_{\bullet})$
be the  weighted biserial quiver 
considered in Example~\ref{ex:2.8}.
Then the  triangulation quiver $(Q^*,f^*)$
is of the form
\[
\begin{tikzpicture}
[->,scale=1.5]

\coordinate (1) at (5,-1);
\coordinate (1-left) at (4.88,-1);
\coordinate (1-up) at (4.88,-.94);
\coordinate (2) at (3,0.7);
\coordinate (3) at (4,0.5);
\coordinate (4) at (2,0);
\coordinate (5) at (0,1);
\coordinate (6) at (0,-1);
\coordinate (7) at (0,-2);
\coordinate (8) at (0,-3);

\coordinate (a) at (0,-3.5);
\coordinate (d) at (4,-0.5);
\coordinate (p) at (0.5,0);

\coordinate (alpha) at (5,1);
\coordinate (alpha-left) at (4.88,1);
\coordinate (alpha-down) at (4.88,.94);
\coordinate (beta) at (3,-0.7);
\coordinate (gamma) at (3.5,0);
\coordinate (delta) at (2,-1);
\coordinate (rho) at (1,1);
\coordinate (sigma) at (4.5,0);
\coordinate (omega) at (2,1);

\coordinate (nu) at (-0.5,0);5
\coordinate (mu) at (1,-1);
\coordinate (xi) at (-0.5,-1.5);
\coordinate (eta) at (0.5,-1.5);
\coordinate (psi) at (-0.75,-2.5);
\coordinate (phi) at (0.75,-2.5);

\coordinate (3l) at (3.94,0.44) ;
\coordinate (3r) at (4.06,0.44) ;
\coordinate (dl) at (3.94,-0.44) ;
\coordinate (dr) at (4.06,-0.44) ;

\coordinate (xi-u) at (-0.44,-1.44) ;
\coordinate (xi-d) at (-0.44,-1.56) ;
\coordinate (eta-u) at (0.44,-1.44) ;
\coordinate (eta-d) at (0.44,-1.56) ;

\coordinate (8ul) at (-0.06,-2.94) ;
\coordinate (8l) at (-0.06,-3.06) ;
\coordinate (al) at (-0.06,-3.44) ;
\coordinate (all) at (-0.06,-3.44) ;

\coordinate (8ur) at (0.06,-2.94) ;
\coordinate (8r) at (0.06,-3.06) ;
\coordinate (ar) at (0.06,-3.44) ;
\coordinate (arr) at (0.06,-3.44) ;

\coordinate (5l) at (-.04,0.92) ;
\coordinate (5r) at (0.08,1) ;
\coordinate (pl) at (0.42,0) ;
\coordinate (pr) at (0.54,0.08) ;

\fill[rounded corners=3mm,fill=gray!20] (1) -- (alpha) -- (sigma) -- cycle;
\fill[fill=gray!20] (1-left) -- (1-up) -- (beta) -- (delta) -- cycle;
\fill[rounded corners=3mm,fill=gray!20] (2) -- (gamma) -- (beta) -- cycle;
\fill[fill=gray!20] (2) -- (omega) -- (alpha-left) -- (alpha-down) -- cycle;
\node [fill=white,circle,minimum size=1.5] at (1) { };
\node [fill=white,circle,minimum size=1.5] at (2) { };
\node [fill=white,circle,minimum size=1.5] at (alpha) { };
\node [fill=white,circle,minimum size=1.5] at (beta) { };
\node [fill=white,circle,minimum size=1.5] at (omega) { };
\node [fill=white,circle,minimum size=1.5] at (delta) { };

\fill[rounded corners=3mm,fill=gray!20] (4) -- (rho) -- (omega) -- cycle;
\fill[rounded corners=3mm,fill=gray!20] (4) -- (delta) -- (mu) -- cycle;
\fill[rounded corners=3mm,fill=gray!20] (6) -- (mu) -- (nu) -- cycle;
\fill[rounded corners=3mm,fill=gray!20] (7) -- (psi) -- (phi) -- cycle;

\fill[rounded corners=3mm,fill=gray!20] (3r) -- (dr) -- (sigma) -- cycle;
\fill[rounded corners=3mm,fill=gray!20] (3l) -- (dl) -- (gamma) -- cycle;

\fill[rounded corners=3mm,fill=gray!20] (6) -- (eta-u) -- (xi-u) -- cycle;
\fill[rounded corners=3mm,fill=gray!20] (7) -- (xi-d) -- (eta-d) -- cycle;

\fill[rounded corners=3mm,fill=gray!20] (psi) -- (8ul) -- (8l) -- (al) -- (all) -- cycle;
\fill[rounded corners=3mm,fill=gray!20] (phi) -- (8ur) -- (8r) -- (ar) -- (arr) -- cycle;

\fill[rounded corners=3mm,fill=gray!20] (5r) -- (pr) -- (rho) -- cycle;
\fill[rounded corners=3mm,fill=gray!20] (5l) -- (nu) -- (pl) -- cycle;

\node (1) at (5,-1) {$\bullet$};
\node (1-up) at (4.95,-.95) { };
\node (2) at (3,0.7) {$\bullet$};
\node (3) at (4,0.5) {$\bullet$};
\node (4) at (2,0) {$\bullet$};
\node (5) at (0,1) {$\bullet$};
\node (6) at (0,-1) {$\bullet$};
\node (7) at (0,-2) {$\bullet$};
\node (8) at (0,-3) {$\bullet$};

\node (a) at (0,-3.5) {$\bullet$};
\node (d) at (4,-0.5) {$\bullet$};
\node (p) at (0.5,0) {$\bullet$};

\node (alpha) at (5,1) {$\bullet$};

\node (alpha-down) at (4.95,.95) { };
\node (beta) at (3,-0.7) {$\bullet$};
\node (gamma) at (3.5,0) {$\bullet$};
\node (delta) at (2,-1) {$\bullet$};
\node (rho) at (1,1) {$\bullet$};
\node (sigma) at (4.5,0) {$\bullet$};
\node (omega) at (2,1) {$\bullet$};

\node (nu) at (-0.5,0) {$\bullet$};
\node (mu) at (1,-1) {$\bullet$};
\node (xi) at (-0.5,-1.5) {$\bullet$};
\node (eta) at (0.5,-1.5) {$\bullet$};
\node (psi) at (-0.75,-2.5) {$\bullet$};
\node (phi) at (0.75,-2.5) {$\bullet$};

\node (3l) at (3.94,0.45) { };
\node (3r) at (4.06,0.45) { };
\node (dl) at (3.94,-0.45) { };
\node (dr) at (4.06,-0.45) { };

\node (xi-u) at (-0.45,-1.44) { };
\node (xi-d) at (-0.45,-1.56) { };
\node (eta-u) at (0.45,-1.44) { };
\node (eta-d) at (0.45,-1.56) { };

\node (8ul) at (-0.06,-2.94) { };
\node (8l) at (-0.06,-3.03) { };
\node (al) at (-0.06,-3.44) { };
\node (all) at (-0.06,-3.44) { };

\node (8ur) at (0.06,-2.94) { };
\node (8r) at (0.06,-3.03) { };
\node (ar) at (0.06,-3.44) { };
\node (arr) at (0.06,-3.44) { };

\node (5l) at (-.04,0.92) { };
\node (5r) at (0.08,1) { };
\node (pl) at (0.42,0) { };
\node (pr) at (0.54,0.08) { };

\draw (1) node[below right]{$1$};
\draw (2) node[below left]{$2$};
\draw (3) node[above]{$3$};
\draw (4) node[right]{$4$};
\draw (5) node[above left]{$5$};
\draw (6) node[left]{$6$};
\draw (7) node[left]{$7$};
\draw (8) node[above]{$8$};

\draw (a) node[below]{$x_a$};
\draw (d) node[below]{$x_d$};
\draw (p) node[below right]{$x_p$};

\draw (alpha) node[above right]{$x_{\alpha}$};
\draw (beta) node[above left]{$x_{\beta}$};
\draw (gamma) node[left]{$x_{\gamma}$};
\draw (delta) node[below]{$x_{\delta}$};
\draw (rho) node[above]{$x_{\varrho}$};
\draw (sigma) node[right]{$x_{\sigma}$};
\draw (omega) node[above]{$x_{\omega}$};

\draw (nu) node[left]{$x_{\nu}$};
\draw (mu) node[below]{$x_{\mu}$};
\draw (xi) node[left]{$x_{\xi}$};
\draw (eta) node[right]{$x_{\eta}$};
\draw (psi) node[left]{$x_{\psi}$};
\draw (phi) node[right]{$x_{\varphi}$};

\draw
(1) edge node[right]{\footnotesize$\alpha'$} (alpha)
(alpha) edge node[left]{\footnotesize$\varepsilon_{\sigma}$} (sigma)
(sigma) edge node[left]{\footnotesize{$\sigma''$}} (1)
(1-up) edge node[above left]{\footnotesize$\!\!\!\beta'\ \ $} (beta)
(beta) edge node[above]{\footnotesize$\varepsilon_{\delta}$} (delta)
(delta) edge node[below]{\footnotesize{$\delta''$}} (1)
(2) edge node[right]{\footnotesize$\gamma'$} (gamma)
(gamma) edge node[right]{\footnotesize$\varepsilon_{\beta}$} (beta)
(beta) edge node[left]{\footnotesize{$\beta''$}} (2)
(2) edge node[below]{\footnotesize$\omega'$} (omega)
(omega) edge node[above]{\footnotesize$\varepsilon_{\alpha}$} (alpha)
(alpha-down) edge node[below left]{\footnotesize{$\!\!\!\alpha''\ \ $}} (2)
(4) edge node[below left]{\footnotesize$\varrho'$} (rho)
(rho) edge node[above]{\footnotesize$\varepsilon_{\omega}$} (omega)
(omega) edge node[right]{\footnotesize{$\omega''$}} (4)
(4) edge node[right]{\footnotesize$\delta'$} (delta)
(delta) edge node[below]{\footnotesize$\varepsilon_{\mu}$} (mu)
(mu) edge node[above left]{\footnotesize{$\mu''$}} (4)
(6) edge node[above left]{\footnotesize$\mu'$} (mu)
(mu) edge node[above right]{\footnotesize$\varepsilon_{\nu}$} (nu)
(nu) edge node[left]{\footnotesize{$\nu''$}} (6)
(7) edge node[above left]{\footnotesize$\psi'$} (psi)
(psi) edge node[below]{\footnotesize$\varepsilon_{\varphi}$} (phi)
(phi) edge node[above right]{\footnotesize{$\varphi''$}} (7)
;
\draw
(3l) edge node[left]{\footnotesize$d'$} (dl)
(dl) edge node[below]{\footnotesize$\varepsilon_{\gamma}$} (gamma)
(gamma) edge node[above]{\footnotesize{$\gamma''$}} (3l)
(3r) edge node[above]{\footnotesize$\sigma'$} (sigma)
(sigma) edge node[below]{\footnotesize$\ \varepsilon_{d}\!\!$} (dr)
(dr) edge node[right]{\footnotesize{$d''$}} (3r)
;
\draw
(6) edge node[right]{\footnotesize$\ \eta'$} (eta-u)
(eta-u) edge node[above]{\footnotesize$\varepsilon_{\xi}$} (xi-u)
(xi-u) edge node[left]{\footnotesize{$\xi''$}} (6)
(7) edge node[left]{\footnotesize$\xi'\ $} (xi-d)
(xi-d) edge node[below]{\footnotesize$\varepsilon_{\eta}$} (eta-d)
(eta-d) edge node[right]{\footnotesize{$\eta''$}} (7)
;
\draw
(8) edge node[above]{\footnotesize$\varphi'$} (phi)
(phi) edge node[below right]{\footnotesize$\varepsilon_{a}$} (arr)
(ar) edge node[above right]{\footnotesize{$\!a''$}} (8r)
(8l) edge node[above left]{\footnotesize$\ \ a'\!\!\!$} (al)
(all) edge node[below left]{\footnotesize$\varepsilon_{\psi}$} (psi)
(psi) edge node[above]{\footnotesize{$\psi''$}} (8)
;
\draw
(5r) edge node[above right]{\footnotesize$\!\!p'$} (pr)
(pr) edge node[below right]{\footnotesize$\varepsilon_{\varrho}$} (rho)
(rho) edge node[above]{\footnotesize{$\varrho''$}} (5r)
(5l) edge node[above left]{\footnotesize$\nu'$} (nu)
(nu) edge node[below]{\footnotesize$\varepsilon_{p}$} (pl)
(pl) edge node[below left]{\footnotesize{$p''\!\!$}} (5l)
;
\end{tikzpicture}
\]
where the shaded triangles describe the $f^*$-orbits in $Q_1^*$.
Then we have the following $g^*$-orbits in $Q_1^*$:
\begin{align*}
  \cO^*(\alpha') & = (\alpha' \ \alpha''
                                \ \gamma' \ \gamma''
                                \ \sigma' \ \sigma''
                                \ \beta' \ \beta''
                                \ \omega' \ \omega''
                                \ \delta' \ \delta''
                                ) , 
 &
  \cO^*(d') & = (d' \ d'' ) , 
 \\
  \cO^*(\mu') & = (\mu' \ \mu''
                                \ \varrho' \ \varrho''
                                \ \nu' \ \nu''
                                \ \eta' \ \eta''
                                \ \psi' \ \psi''
                                \ \varphi' \ \varphi''
                                \ \xi' \ \xi''
                                ) , 
 &
  \cO^*(p') & = (p' \ p'' ) , 
 \\
  \cO^*(\varepsilon_{\alpha}) & = (\varepsilon_{\alpha}
                                \ \varepsilon_{\sigma} 
                                \ \varepsilon_{d} 
                                \ \varepsilon_{\gamma} 
                                \ \varepsilon_{\beta} 
                                \ \varepsilon_{\delta} 
                                \ \varepsilon_{\mu} 
                                \ \varepsilon_{\nu} 
                                \ \varepsilon_{p} 
                                \ \varepsilon_{\varrho} 
                                \ \varepsilon_{\omega} 
                                ) ,
 &
  \cO^*(a') & = (a' \ a'' ) , 
 \\
    \cO^*(\varepsilon_{a}) & = (\varepsilon_{a}
                                \ \varepsilon_{\psi} 
                                \ \varepsilon_{\varphi} 
                                ) ,
 &
    \cO^*(\varepsilon_{\eta}) & = (\varepsilon_{\eta}
                                \ \varepsilon_{\xi} 
                                )  . 
\end{align*}
Moreover, the weight function $m_{\bullet}^* : \cO(g^*) \to \bN^*$
is given by
\begin{align*}
 m_{\cO^*(d')}^* &= m_{\cO(d)} = e(d),
&
 m_{\cO^*(\alpha')}^* &= m_{\cO(\alpha)} = e(c),
&
 m_{\cO^*(\varepsilon_{\alpha})}^* &= 1,
\\
 m_{\cO^*(p')}^* &= m_{\cO(p)} = e(p),
&
 m_{\cO^*(\mu')}^* &= m_{\cO(\mu)} = e(b),
&
 m_{\cO^*(\varepsilon_{\eta})}^* &= 1,
\\
&
&
 m_{\cO^*(a')}^* &= m_{\cO(a)} = e(a).
&
\end{align*}
Finally,
$e^* = e_1^* + e_2^* + e_3^* + e_4^* + e_5^* + e_6^* + e_7^* + e_8^*$. 
We note that $(Q^*,f^*)$ has $16$ $f^*$-orbits, all of length three.

The  Brauer graph algebra
$B_{\Gamma} = B(Q,f,m_{\bullet})$
is also isomorphic to the  idempotent algebra
$e' B' e'$ of a biserial triangulation 
algebra $B' = B(Q',f',m'_{\bullet})$
for  the triangulation quiver $(Q', f')$ shown 
below
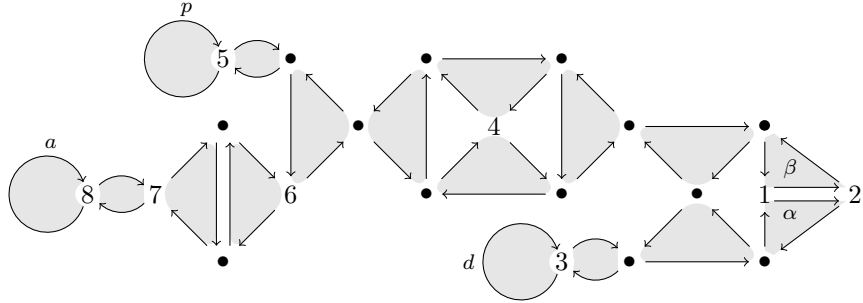
\[
\begin{tikzpicture}
[->,scale=.6]

\coordinate (1) at (0,0) ;
\coordinate (1u) at (0.15,0.15) ;
\coordinate (1d) at (0.15,-0.15) ;
\coordinate (2) at (3,0) ;
\coordinate (2u) at (2.85,0.15) ;
\coordinate (2d) at (2.85,-0.15) ;

\coordinate (3) at (1,-3) ;
\coordinate (4) at (0,0) ;
\coordinate (5) at (-6,1.5) ;
\coordinate (6) at (-4.5,-1.5) ;
\coordinate (7) at (-7.5,-1.5) ;
\coordinate (8) at (-9,-1.5) ;

\coordinate (a1) at (-1.5,1.5) ;
\coordinate (b1) at (1.5,1.5) ;
\coordinate (c1) at (3,0) ;
\coordinate (d1) at (6,0) ;
\coordinate (6') at (-4.5,1.5) ;
\coordinate (e) at (-3,0) ;
\coordinate (f) at (6,0) ;
\coordinate (a2) at (-1.5,-1.5) ;
\coordinate (b2) at (1.5,-1.5) ;
\coordinate (c2) at (3,-3) ;
\coordinate (d2) at (6,-3) ;
\coordinate (g) at (4.5,-1.5) ;
\coordinate (u) at (-6,0) ;
\coordinate (ul) at (-6.15,-0.15) ;
\coordinate (ur) at (-5.85,-0.15) ;
\coordinate (d) at (-6,-3) ;
\coordinate (dl) at (-6.15,-2.85) ;
\coordinate (dr) at (-5.85,-2.85) ;

\fill[fill=gray!20] 
   (6,-1.35) -- (7.8,-1.35) -- (6.2,-.15) -- (6,-.25) -- cycle;
\fill[fill=gray!20] 
   (6,-1.65) -- (7.8,-1.65) -- (6.2,-2.85) -- (6,-2.75) -- cycle;

\fill[rounded corners=3mm,fill=gray!20] (e) -- (6') -- (6) -- cycle;
\fill[rounded corners=3mm,fill=gray!20] (a1) -- (b1) -- (4) -- cycle;

\fill[rounded corners=3mm,fill=gray!20] (e) -- (a2) -- (a1) -- cycle;
\fill[rounded corners=3mm,fill=gray!20] (a2) -- (b2) -- (4) -- cycle;

\fill[rounded corners=3mm,fill=gray!20] (7) -- (ul) -- (dl) -- cycle;
\fill[rounded corners=3mm,fill=gray!20] (dr) -- (ur) -- (6) -- cycle;

\fill[rounded corners=3mm,fill=gray!20] (b1) -- (b2) -- (c1) -- cycle;

\fill[rounded corners=3mm,fill=gray!20] (g) -- (c1) -- (d1) -- cycle;
\fill[rounded corners=3mm,fill=gray!20] (g) -- (c2) -- (d2) -- cycle;

\fill[fill=gray!20] (-5.75,1.7) arc (135:45:0.7) -- (-4.75,1.3) arc (315:225:0.7) -- cycle;
\fill[fill=gray!20] (-8.75,-1.3) arc (135:45:0.7) -- (-7.75,-1.7) arc (315:225:0.7) -- cycle;
\fill[fill=gray!20] (1.75,-2.7) arc (135:45:0.7) -- (2.75,-3.2) arc (315:225:0.7) -- cycle;

\node [circle,minimum size=1.5](A) at (-6,1.5) { };
\node [circle,minimum size=1cm](B) at (-6.9,1.5) {};
\coordinate  (C) at (intersection 2 of A and B);
\coordinate  (D) at (intersection 1 of A and B);
 \tikzAngleOfLine(B)(D){\AngleStart}
 \tikzAngleOfLine(B)(C){\AngleEnd}
\fill[gray!20]%
   let \p1 = ($ (B) - (D) $), \n2 = {veclen(\x1,\y1)}
   in   
     (D) arc (\AngleStart:\AngleEnd:\n2); 
\node [fill=white,circle,minimum size=1.5](A) at (-6,1.5) { };
\draw[<-]%
   let \p1 = ($ (B) - (D) $), \n2 = {veclen(\x1,\y1)}
   in   
     (B) ++(60:\n2) node[left]{\footnotesize\ \ \ \raisebox{4ex}{$p$}}
     (D) arc (\AngleStart:\AngleEnd:\n2); 
\node [circle,minimum size=1.5](A) at (-9,-1.5) { };
\node [circle,minimum size=1cm](B) at (-9.9,-1.5) {};
\coordinate  (C) at (intersection 2 of A and B);
\coordinate  (D) at (intersection 1 of A and B);
 \tikzAngleOfLine(B)(D){\AngleStart}
 \tikzAngleOfLine(B)(C){\AngleEnd}
\fill[gray!20]%
   let \p1 = ($ (B) - (D) $), \n2 = {veclen(\x1,\y1)}
   in   
     (D) arc (\AngleStart:\AngleEnd:\n2); 
\node [fill=white,circle,minimum size=1.5](A) at (-9,-1.5) { };
\draw[<-]%
   let \p1 = ($ (B) - (D) $), \n2 = {veclen(\x1,\y1)}
   in   
     (B) ++(60:\n2) node[left]{\footnotesize\ \ \ \raisebox{4ex}{$a$}}
     (D) arc (\AngleStart:\AngleEnd:\n2); 

\node [circle,minimum size=1.5](A) at (1.5,-3) { };
\node [circle,minimum size=1cm](B) at (.6,-3) {};
\coordinate  (C) at (intersection 2 of A and B);
\coordinate  (D) at (intersection 1 of A and B);
 \tikzAngleOfLine(B)(D){\AngleStart}
 \tikzAngleOfLine(B)(C){\AngleEnd}
\fill[gray!20]%
   let \p1 = ($ (B) - (D) $), \n2 = {veclen(\x1,\y1)}
   in   
     (D) arc (\AngleStart:\AngleEnd:\n2); 
\node [fill=white,circle,minimum size=1.5](A) at (1.5,-3) { };
\draw[<-]%
   let \p1 = ($ (B) - (D) $), \n2 = {veclen(\x1,\y1)}
   in   
     (B) node[left]{\footnotesize\raisebox{0ex}{\!\!\!\!\!$d$\ \ \quad}} ++(60:\n2)
     (D) arc (\AngleStart:\AngleEnd:\n2); 

\node (1) at (6,-1.5) {1};
\node (1u) at (6.,-1.35) { };
\node (1d) at (6.,-1.65) { };
\node (2) at (8,-1.5) {2};
\node (2u) at (7.85,-1.35) { };
\node (2d) at (7.85,-1.65) { };

\node (3) at (1.5,-3) {3};
\node (4) at (0,0) {4};
\node (5) at (-6,1.5) {5};
\node (6) at (-4.5,-1.5) {6};
\node (7) at (-7.5,-1.5) {7};
\node (8) at (-9,-1.5) {8};

\node (a1) at (-1.5,1.5) {$\bullet$};
\node (b1) at (1.5,1.5) {$\bullet$};
\node (c1) at (3,0) {$\bullet$};
\node (d1) at (6,0) {$\bullet$};
\node (6') at (-4.5,1.5) {$\bullet$};
\node (e) at (-3,0) {$\bullet$};
\node (f) at (6,0) {$\bullet$};
\node (a2) at (-1.5,-1.5) {$\bullet$};
\node (b2) at (1.5,-1.5) {$\bullet$};
\node (c2) at (3,-3) {$\bullet$};
\node (d2) at (6,-3) {$\bullet$};
\node (g) at (4.5,-1.5) {$\bullet$};

\node (u) at (-6,0) {$\bullet$};
\node (ul) at (-6.15,-0.15) { };
\node (ur) at (-5.85,-0.15) { };
\node (d) at (-6,-3) {$\bullet$};
\node (dl) at (-6.15,-2.85) { };
\node (dr) at (-5.85,-2.85) { };

\draw (-5.75,1.7) arc (135:45:0.7);
\draw (-4.75,1.3) arc (315:225:0.7);

\draw (-8.75,-1.3) arc (135:90:0.7) arc (90:45:0.7);
\draw (-7.75,-1.7) arc (315:270:0.7) arc (270:225:0.7);

\draw (1.75,-2.7) arc (135:45:0.7);
\draw (2.75,-3.2) arc (315:225:0.7);

\draw
(e) edge (6')
(6') edge (6) 
(6) edge (e) 
(e) edge (a2) 
(a2) edge (a1) 
(a1) edge (e) 
(a1) edge (b1) 
(b1) edge (4)
(4) edge (a1)
(b2) edge (a2) 
(a2) edge (4)
(4) edge (b2)
(b1) edge (b2) 
(b2) edge (c1) 
(c1) edge (b1) 
(c1) edge (d1) 
(d1) edge (g) 
(g) edge (c1) 
(c2) edge (d2) 
(d2) edge (g) 
(g) edge (c2) 
(d1) edge (1) 
(d2) edge (1) 
(2) edge (d1) 
(2) edge (d2) 
(1u) edge node[above left]{\footnotesize$\beta$} (2u) 
(1d) edge node[below left]{\footnotesize$\alpha$} (2d) 
;
\draw
(7) edge (ul) 
(ul) edge (dl)
(dl) edge (7)
(dr) edge (ur) 
(ur) edge (6)
(6) edge (dr)
;
\end{tikzpicture}
\]
with $14$ $f'$-orbits described by the shaded triangles
(all of length three),
a weight function $m'_{\bullet}$ 
of $(Q', f')$,
and where the idempotent $e'$
is the sum of the primitive idempotents in $B'$
associated to the vertices $1,2,3,4,5,6,7,8$.
\end{example}

We finish this section with a 
 combinatorial interpretation
of the $*$-construction in terms of Brauer graphs.

\addtocounter{subsection}{7}
\subsection{Barycentric division of Brauer graphs.}
\label{4.8}

Let $\Gamma$ be the Brauer graph so that $B_{\Gamma} = B(Q, f, m_{\bullet})$, then the
algebra $B(Q^*, f^*, m_{\bullet}^*)$ as in the $*$-construction of
Theorem \ref{th:4.1} is again a Brauer graph algebra, $B_{\Gamma^*}$ say, by Theorem \ref{th:2.6}.
The proof of Lemma \ref{lem:2.4} shows how to construct 
$\Gamma^*$:  Its vertices
are in bijection with the cycles of $g^*$.  First, each cycle of $g$ is 
\lq augmented\rq, by replacing an arrow $\alpha$ by $\alpha', \alpha''$, 
and this gives a cycle of $g^*$, we call a corresponding  vertex
of $\Gamma^*$ an augmented vertex.
Second, any other cycle of $g^*$ consists of $\ve$-arrows, 
and these cycles correspond to $f$-cycles of $Q$, as
described in Theorem \ref{th:4.1}. 
Let $F(\alpha)$ be the $f$-orbit of $\alpha$ in $Q$, 
then we write $v_{F(\alpha)}$ for the corresponding vertex of $\Gamma^*$, 
then the arrows attached to
this vertex are precisely the $\ve_{f^t(\alpha)}$.

The edges of $\Gamma^*$ are labelled by the vertices of $Q^*$, that is by
the vertices of $Q$ together with  the set $\{ x_{\alpha} \mid \alpha \in Q_1\}$.
The cyclic order around an augmented vertex is obtained by replacing 
$i\stackrel{\alpha}\longrightarrow j$ 
by
$$i\stackrel{\alpha'}\longrightarrow x_{\alpha}\stackrel{\alpha''}\longrightarrow j$$
in $\Gamma^*$.
A  vertex $v_{F(\alpha)}$ has attached arrows precisely 
the $\ve_{f^t(\alpha)}: x_{f^{t+1}(\alpha)}\to  x_{f^t(\alpha)}$. 
This specifies
 the edges adjacent, with   cyclic order given by
the inverse of the $f$-cycle of $\alpha$.
We may view $\Gamma^*$ as a \lq triangular\rq{}  graph: 


(1)
Assume that $|F(\alpha)| = 1$.
Then  $x_{\alpha}$ is the unique
edge in $\Gamma^*$ adjacent to $v_{F(\alpha)}$, and $x_{\alpha}$
is  its own successor in the cyclic order
of edges in $\Gamma^*$ around $v_{F(\alpha)}$.
Hence we have
in $\Gamma^*$ a self-folded triangle 
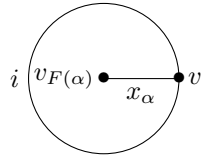
\[
\begin{tikzpicture}[auto]
\coordinate (c) at (0,0);
\coordinate (a) at (1,0);
\coordinate (b) at (-1,0);
\draw (a) node [right] {$v$} to node {$x_{\alpha}$} (c) node [left] {$v_{F(\alpha)}$};
\draw (b) arc (-180:180:1) node [left] {$i$};
\node (a) at (1,0) {$\bullet$};
\node (c) at (0,0) {$\bullet$};
\end{tikzpicture}
\]
which corresponds to a subquiver of  $(Q^*,f^*)$  of the form
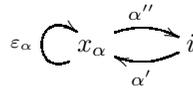
\[
\xymatrix{ 
	x_{\alpha} \ar@(dl,ul)[]^{\varepsilon_{\alpha}} \ar@/^1.5ex/[r]^{\alpha''}
	& i \ar@/^1.5ex/[l]^{\alpha'}
}
\]
with $f^*$-orbit
$(\alpha'' \ \alpha' \ \varepsilon_{\alpha})$.

\smallskip

(2)
Assume that $|F(\alpha)| \geq 2$, and let 
$\beta = f(\alpha)$ starting at vertex $j$.
Let $v, w$ be the vertices in $\Gamma$ such that
$\alpha$ is attached to $v$ and $\beta$ is attached to $w$. 
Then 
$\Gamma^*$ has edges $x_{\alpha}$ and
$x_{\beta}$ connecting vertices $v$ and $w$ to vertex $v_{F(\beta)} (= v_{F(\alpha)})$.
Then $x_{\alpha}$ 
is the  successor of $x_{\beta}$
in the cyclic order of edges in $\Gamma^*$ 
around $v_{F(\alpha)}$.
Hence we have in $\Gamma^*$ a triangle
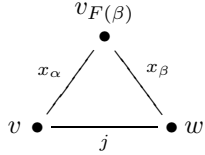
\[
\xymatrix@C=1.2pc{ 
	& \bullet \ar@{-}[ld]_{x_{\alpha}} 
	  \ar@{-}[rd]^{x_{\beta}} 
      \save[] +<0mm,3mm> *{v_{F(\beta)}} \restore 
	  \\
    \save[] +<-3mm,0mm> *{v} \restore 
	\bullet \ar@{-}[rr]_{j} && \bullet 
    \save[] +<+3mm,0mm> *{w} \restore 
}
\]
which corresponds to a subquiver of $(Q^*,f^*)$  of the form
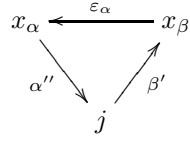
\[
\xymatrix@C=1.2pc{ 
	x_{\alpha} \ar[rd]_{\alpha''} 
	  && x_{\beta} \ar[ll]_{\varepsilon_{\alpha}} \\ 
	& j \ar[ru]_{\beta'}
}
\]
with $f^*$-orbit
$(\alpha'' \ \beta' \ \varepsilon_{\alpha})$.
The multiplicity function $e^*$ of $\Gamma^*$
is given by $e^*(v) = e(v)$ for any vertex $v$
of $\Gamma$ (where $e$ is the multiplicity function for $\Gamma$), 
 and $e^*(v_{F(\alpha)})=1$ for any
$f$-orbit $F(\alpha)$.

The Brauer graph $\Gamma^*$
can be considered as a \emph{barycentric division}
of the Brauer graph $\Gamma$,
and has a triangular structure.
Namely, every $v_{F(\alpha)}$
is the vertex of $|F(\alpha)|$
triangles in $\Gamma^*$ whose edges opposite
to $v_{F(\alpha)}$ are the edges of $\Gamma$ 
corresponding to the vertices in $Q$ along $F(\alpha)$.

In this way, we obtain an orientable surface $S^*$
without boundary, the triangulation $T^*$ of $S^*$
indexed by the set of edges of $\Gamma$, 
and the orientation $\vv{T^*}$
of triangles in $T^*$ such that the associated
triangulation quiver $(Q(S^*,\vv{T^*}),f)$
is the quiver $(Q^*,f^*)$.
The triangulated surface $(S^*,T^*)$
can be considered as a completion of the
Brauer graph $\Gamma$ to a canonically defined
triangulated surface, by a finite number of pyramids
whose peaks are the $f$-orbits
and bases are given by the edges of $\Gamma$.
We also note that the surface $S^*$
(without triangulation $T^*$)
can be obtained as follows.
We may embed the Brauer graph $\Gamma$
into a surface $S$ 
with boundary given by thickening 
the edges of $\Gamma$. 
The components of the border $\partial S$ of $S$
are given by the 
\lq Green walks\rq{}  
around $\Gamma$ on $S$,
corresponding to the $f$-orbits in $Q_1$.
Then the surface $S^*$ is obtained from $S$
by capping all the boundary components of $S$
by the disks $D^2$.

\addtocounter{theorem}{1}

\begin{example}
\label{ex:4.10}
Let $\Gamma$ be the Brauer graph
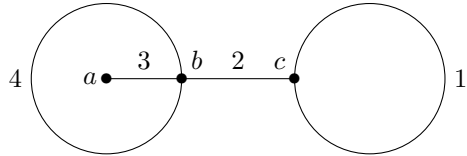
\[
\begin{tikzpicture}[auto]
 \coordinate (c) at (0,0);
 \coordinate (a) at (1,0);
 \coordinate (b) at (-1,0);
 \coordinate (d) at (2.5,0);
 \coordinate (e) at (4.5,0);
 \draw (c) node[left]{$a$} to node {$3$} (a);
 \draw (a) node[above right]{$b$} to node {$2$} (d) node[above left]{$c$};
 \draw (b) arc (-180:180:1) node [left] {$4$};
 \draw (e) arc (0:360:1) node [right] {$1$};
 \node (a) at (1,0) {$\bullet$};
 \node (c) at (0,0) {$\bullet$};
 \node (d) at (2.5,0) {$\bullet$};
 \end{tikzpicture}
\]
where we take the clockwise ordering of edges around each vertex.
Assume  
the multiplicity function takes only value $1$.
Then the associated biserial quiver $(Q,f)$
is of the form
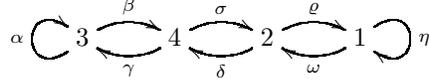
\[
  \xymatrix{ 
     3 \ar@(dl,ul)[]^{{\alpha}} \ar@/^1.5ex/[r]^{\beta}
       & 4 \ar@/^1.5ex/[l]^{\gamma} \ar@/^1.5ex/[r]^{\sigma}
       & 2 \ar@/^1.5ex/[l]^{\delta} \ar@/^1.5ex/[r]^{\varrho}
       & \ar@/^1.5ex/[l]^{\omega}
       1 \ar@(ur,dr)[]^{{\eta}} 
  }
\]
with 
$f$-orbits
\begin{align*}
  F(\alpha) &= (\alpha\ \beta\ \gamma), & 
  F(\sigma) &= (\sigma\ \varrho\ \omega\ \delta), & 
  F(\eta) &= (\eta), 
\end{align*}
and $\cO(g)$ consisting of
\begin{align*}
  \cO(\alpha) &= (\alpha), & 
  \cO(\beta) &= (\beta\ \sigma\ \delta\ \gamma), & 
  \cO(\omega) &= (\omega\ \varrho\ \eta).
\end{align*}
Then the barycentric division $\Gamma^*$ of $\Gamma$
is the  Brauer graph 
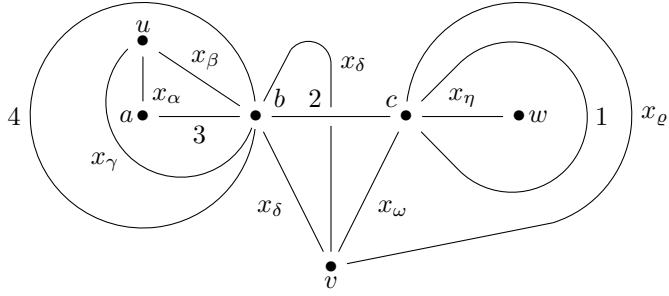
\[
\begin{tikzpicture}[auto]
 \node (o) at (0,0) { };
 \node (a) at (-2.5,0) {$\bullet$};
 \node (b) at (-1,0) {$\bullet$};

 \node (c) at (1,0) {$\bullet$};
 \node (u) at (-2.5,1) {$\bullet$};
 \node (v) at (0,-2) {$\bullet$};
 \node (w) at (2.5,0) {$\bullet$};
\draw (a) edge node [below] {$3$} (b);
\draw (a) edge node[below right] {$x_{\alpha}$}(u);
\draw (b) edge node[above left] {$2$} (c);
\draw (b) edge node[above] {$\ x_{\beta}\!$} (u);
\draw (v) edge  node {$x_{\delta}$} (b);
\draw (c) edge node {$x_{\eta}$} (w);
\draw (c) edge node {$x_{\omega}$} (v);
\draw (o) edge (v);
\draw (-1.01,0.2) arc (7:353:1.5);
\draw (1.01,0.2) arc (173:-72.5:1.5) edge (v);
\draw (o) edge (0,.7);
\draw (0,.7) arc (0:142:.3) edge (b);
\draw (-2.7,0.9) arc (135:340:1);
\draw (c) edge (1.7,.7);
\draw (1.7,.7) arc (135:-135:1) edge (c);
\draw (0,.7) node [right] {$x_{\delta}$};
\draw (w) node [right] {$w$};
\draw (a) node [left] {$a$};
\draw (b) node [above right] {$\ b\!$};
\draw (c) node [above left] {$c$};
\draw (u) node [above] {$u$};
\draw (v) node [below] {$v$};
\draw (4,0) node [right] {$x_{\varrho}$};
\draw (-4,0) node [left] {$4$};
\draw (3.6,0) node {$1$};
\draw (-3,-.6) node {$x_{\gamma}$};
 \end{tikzpicture}
\]
with 
$u = v_{F(\alpha)}$,
$v = v_{F(\sigma)}$ and  
$w = v_{F(\eta)}$.
The ordering of the edges around each vertex is clockwise.
The multiplicity function of $\Gamma^*$
takes only the value $1$.

The Brauer graph $\Gamma$ admits a canonical
embedding into the surface $S$ of the form
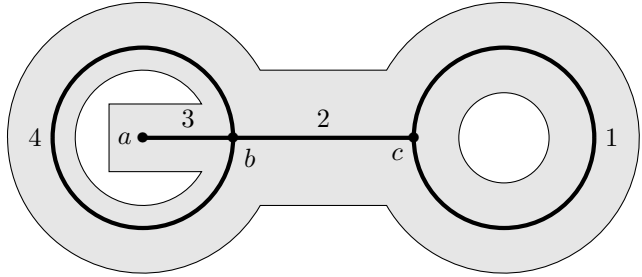
\[
\begin{tikzpicture}[auto,scale=1.2]
 \coordinate (o) at (-1.5,0);
 \coordinate (a) at (0,0);
 \coordinate (b) at (1,0);
 \coordinate (b1) at (-1,0); 
 \coordinate (c) at (3,0);
 \coordinate (e) at (5,0);
 
 \fill[fill=gray!20] (o) arc (180:30:1.5) -- (2.7,.75) arc (150:-150:1.5) -- (1.3,-.75)  arc (-30:-180:1.5)
          (4.5,0) arc (0:360:.5)
          (0.65,.375)  arc (30:330:.75) -- (-.375,-.375) -- (-.375,.375)  --  (0.65,.375)
           ;
 \draw[ultra thick] (a) node[left]{$a$} to node {$3$} (b);
 \draw[ultra thick] (b) node[below right]{$b$} to node {$2$} (c) node[below left]{$c$};
 \draw[ultra thick] (b1) arc (-180:180:1) node [left] {$4$};
 \draw[ultra thick] (e) arc (0:360:1) node [right] {$1$};
 \draw (o) arc (180:30:1.5) -- (2.7,.75) arc (150:-150:1.5) -- (1.3,-.75)  arc (-30:-180:1.5);
 \draw (0.65,-.375)  arc (330:30:.75) -- (-.375,.375) -- (-.375,-.375)  -- cycle;
 \draw (4.5,0) arc (0:360:.5);
 \node (a) at (1,0) {$\bullet$};
 \node (b) at (0,0) {$\bullet$};
 \node (c) at (3,0) {$\bullet$};
 \end{tikzpicture}
\]
obtained from $\Gamma$ by thickening the edges of $\Gamma$,
whose border $\partial S$ has three components given by three different 

\lq Green walks\rq{}  
around $\Gamma$ on $S$.
The triangulated surface $(S^*,T^*)$
associated to the Brauer graph $\Gamma^*$ can be viewed as a canonical
completion of $S$ to a triangulated surface.
\end{example}

\section{Proof of Theorem \ref{th:main3}}\label{sec:proof3}

This theorem describes algebras socle equivalent 
to Brauer graph algebras. 
By Theorem \ref{th:2.6}  
this is the same as describing algebras
socle equivalent to a biserial quiver algebra 
$A=B(Q, f, m_{\bullet})$ where $(Q, f)$ is a biserial quiver.
We show that such algebras can be described 
using the methods of \cite[Section~6]{ESk4}. 
Then  we show that the 
$*$-construction for the biserial quiver algebras can be extended.

\bigskip

Let $(Q,f)$ be a biserial quiver.
A vertex $i \in Q_0$ is said to be a \emph{border vertex}
of $(Q, f)$  if there is a loop $\alpha$ at $i$ with $f(\alpha) = \alpha$.
We denote by $\partial(Q,f)$
the set of all border vertices of $(Q,f)$,
and call it the \emph{border} of $(Q,f)$.
The terminology is motivated by the connection with surfaces: If
$(Q, f)$
is the triangulation quiver
$(Q(S,\vv{T}),f)$
associated to a directed triangulated surface 
$(S,\vv{T})$,
then the border vertices of $(Q,f)$
correspond bijectively to the boundary edges
of the triangulation $T$ of $S$. 
If $(Q, f)$ is the biserial quiver
associated to a Brauer graph $\Gamma$,
then the border vertices of  $(Q, f)$
correspond bijectively to the  internal   loops of  $\Gamma$
(see Section~\ref{sec:bisalg}).

\begin{definition}\label{def:5.1} 
Assume  $(Q, f, m_{\bullet} )$ is a biserial quiver
with $\partial(Q,f)$ not empty.
A function
\[
  b_{\bullet} : \partial(Q,f) \to K
\]
is said to be a \emph{border function} of $(Q,f)$.
We have the quotient algebra
\[
  B(Q,f,m_{\bullet},b_{\bullet})
   = K Q / J(Q,f,m_{\bullet},b_{\bullet}),
\]
where $J(Q,f,m_{\bullet},b_{\bullet})$
is the ideal in the path algebra $KQ$ 
generated by the elements:
\begin{enumerate}[(1)]
 \item
  $\alpha f({\alpha})$,
  for all arrows $\alpha \in Q_1$ which are not border loops,
 \item
  $\alpha^2 - b_{s(\alpha)} B_{\alpha}$,
  for all border loops $\alpha \in Q_1$,
 \item
  $B_{\alpha}
   - B_{\bar{\alpha}}$,
  for all arrows $\alpha \in Q_1$.
\end{enumerate}
We call such an algebra a \emph{biserial quiver algebra with border}.   
Note that if $b_{\bullet}$ is the zero function then
$B(Q,f,m_{\bullet},b_{\bullet}) = B(Q,f,m_{\bullet})$.
\end{definition}

We summarize the basic properties of these algebras.

\begin{proposition}
\label{prop:5.2} 
Let $(Q,f)$ be a biserial quiver
such that $\partial(Q,f)$ is
not empty,
and let 
$\bar{B} = B(Q,f,m_{\bullet},b_{\bullet})$,
and
$B = B(Q,f,m_{\bullet})$ where $m_{\bullet}$ and $b_{\bullet}$ are
weight and border functions.
Then the following statements hold.
\begin{enumerate}[(i)]
 \item
  $\bar{B}$ is a basic, indecomposable, finite-dimensional,
  symmetric, biserial algebra 
  with $\dim_K \bar{B} = \sum_{\cO \in \cO(g)} m_{\cO} n_{\cO}^2$.
 \item
  $\bar{B}$ is socle equivalent to $B$.
 \item
  If $K$ is of characteristic different from $2$,
  then $\bar{B}$ is isomorphic to $B$.
\end{enumerate}
\end{proposition}

\begin{proof} 
Part (ii) is clear from the definition and then part
(i) follows from Proposition \ref{prop:2.3}.  
For the  last part
see  arguments  in the proof of
Proposition 6.3 in \cite{ESk4}.
\end{proof}

The following theorem gives a complete description
of symmetric algebras socle equivalent to a biserial
quiver algebra.

\begin{theorem}
        \label{th:5.3}
        Let $A$ be a basic, indecomposable,
        symmetric algebra with Gro\-t\-hen\-dieck group $K_0(A)$
        of rank at least $2$.
        Assume that $A$ is socle equivalent to a biserial
        quiver algebra $B(Q,f,m_{\bullet})$.
        \begin{enumerate}
                \item If $\partial(Q, f)$ is empty then $A$ is isomorphic to
                        $B(Q,f,m_{\bullet})$. 
                \item
                       Otherwise $A$ is isomorphic to 
                        $B(Q,f,m_{\bullet},b_{\bullet})$
                        for some
                        border function $b_{\bullet}$ of $(Q,f)$.
        \end{enumerate}
\end{theorem}

\begin{proof} 
Let $B= B(Q, f, m_{\bullet}) = KQ/J$ where $J = J(Q, f, m_{\bullet})$. 
Since
$A/\soc(A)$ is isomorphic to $B/\soc(B)$, 
we can assume that these are equal, 
using an isomorphism as identification.
We assume $A$ is symmetric, therefore for each $i \in Q_0$, 
the module $e_iA$ has a 1-dimensional socle which is spanned by some
$\omega_i \in e_iAe_i$, and we fix such an element.
Then let $\vf$ be a symmetrizing linear form for $A$, 
then $\vf(\omega_i)$ is non-zero.  
We may assume that $\vf(\omega_i) = 1$.

We claim that $\soc (A) \subset (\rad A)^2 (\subset \rad A)$.
If not, then for some $j$ we have $\omega_j \not\in (\rad A)^2$.
This means that $e_jA = e_jAe_j$, which is not possible since $A$ is
indecomposable with at least two simple modules.
It follows that $A$ and $B$ have the same Gabriel quiver.
Recall that the quiver $Q$
is the disjoint union of the Gabriel quiver of $B$ with virtual loops.
Any virtual loop of $Q$ is then in the socle of $B$ and it is zero in
$B/\soc (B)$ and is therefore zero in $A/\soc (A)$.
We may  therefore take $A$ of the form $A=KQ/I$ for the same quiver
$Q$, and some ideal $I$ of $KQ$, and such that any virtual loop lies
in the socle of $A$.

\bigskip

In the  algebra $B$ we define monomials $A_{\alpha}$ 
in the arrows by setting $B_{\alpha} = A_{\alpha}g^{-1}(\alpha)$ 
when $\alpha$ is not a virtual loop, and then 
as well $B_{\alpha} = \alpha A_{g(\alpha)}$. 
Note that if $\alpha$ is a virtual loop then $A_{\alpha}$ is not
defined.
With this, the  elements $A_{\alpha}$ belong to the 
socle  of $B/\soc(B)$ and hence also to  the socle of
$A/\soc (A)$. 
Therefore they  cannot lie in the socle of $A$ 
(because if so then they would be zero in $A/\soc (A)$).
Then $0\neq A_{\alpha}\rad(A) = \soc (e_iA)$ where $i=s(\alpha)$.
We have that $A_{\alpha} \rad(A)$ is spanned by
$$A_{\alpha}\beta, A_{\alpha}\gamma$$
where $\beta = g^{-1}(\alpha)$ and $\gamma = f(g^{-2}(\alpha))$.

\bigskip

(I) We may assume that  $A_{\alpha}\beta = B_{\alpha}$  in $A$
(and hence is equal to $B_{\alpha}$ in $KQ$).

If not, then we have  $A_{\alpha}\beta = 0$, and then  
$A_{\alpha}\gamma\neq 0$.
We will show that we may interchange $\beta$ and $\gamma$.

Since $A_{\alpha}\gamma \neq 0$,  in particular 
$g^{-2}(\alpha)\gamma \neq 0$  and also $t(\gamma) = i = s(\alpha)$.
Since $\gamma = f(g^{-2}(\alpha))$ we know that 
$g^{-2}(\alpha)\gamma$ belongs to the socle of $A$. 
It is non-zero, which implies that 
$A_{\alpha} = g^{-2}(\alpha)$ (and $m_{\alpha}=1$), 
and therefore $\alpha = g^{-2}(\alpha)$, and  $\gamma = f(\alpha)$.
We claim that $g(\alpha)\neq \alpha$. 
Namely if we had $g(\alpha)=\alpha$ 
then both $\alpha$ and $f(\alpha)$ would be loops at vertex $i$ and
$|Q_0|=1$, which contradicts our assumption.
Hence the cycle of $g$ containing $\alpha$ is $(\alpha \ g(\alpha))$, of length two. 
We claim that also the $f$-cycle of $\alpha$ (in $B$) has length two.
Namely if $\ba$ is the other arrow starting at $i$ and $\rho$ 
is the other arrow ending at $j=t(\alpha)$ then we must have
by the properties of $f$ and $g$ that $f(\rho) = \beta$ and $f(\beta) = \ba$. 
This implies that $f(\gamma) = \alpha$ and hence
$f$ has a cycle $(\alpha \ \gamma)$.

It follows that 
there is an algebra isomorphism from $B$  
to the biserial quiver algebra $B'$ given by the weighted
biserial quiver obtained from $(Q,f,m_{\bullet})$
by interchanging
$\beta$ and $\gamma$ (which form a pair of double arrows) and fixing
all other arrows of $Q$. 
We replace $B$ by  $B'$ and the claim follows.

\bigskip

(II) We show that relation (1) holds in $A$. 
If $\alpha$ is a virtual loop of $B$ then $\alpha f(\alpha)=0$ since
$\alpha \in \soc (A)$. 
We consider now an arrow $\alpha$ which is not
a virtual loop. Suppose $\alpha$ is not fixed by $f$, then 
 $\alpha f(\alpha)$ belongs to the socle of $A$. 
We can write 
$\alpha f(\alpha) = a_{\alpha} B_{\alpha} = a_{\alpha}\alpha A_{g(\alpha)}$ 
for some $a_{\alpha} \in K$
(here $g(\alpha)$ is not a virtual loop). 

(a) If $s(\alpha)\neq t(f(\alpha))$ then 
$\alpha f(\alpha) = \alpha f(\alpha) e_{s(\alpha)}= 0$, 
in fact this holds for any choice of $\alpha, f(\alpha)$.

(b) Otherwise, we set
$$f(\alpha)':= f(\alpha) - a_{\alpha} A_{g(\alpha)}$$
and we replace $f(\alpha)$ by $f(\alpha)'$.
(If a cycle of $f$ has a virtual loop then $\alpha f(\alpha)$ and 
$f^{-1}(\alpha)\alpha$ are not cyclic paths, so they are zero and do not need
adjusting.)
These 
modifications must be iterated. Take   a cycle of $f$, 
say it has length $r$, so that
$r\geq 2$.

Assume first this cycle contains an arrow $\alpha$ such that $f^{r-1}(\alpha)\alpha$ is not a cyclic path. We may start with $\alpha$ and adjust
$f(\alpha), f^2(\alpha), \ldots , f^{r-1}(\alpha)$ as described above. Then $f^{r-1}(\alpha)'\cdot \alpha =0$, by (a) above.

Otherwise, 
 for any $\alpha$ in the cycle, $f^{r-1}(\alpha)\alpha$ is cyclic, and
 then we must have $r=2$ or $r=4$.
Assume that $r=2$. 
We adjust $f(\alpha)$ as described in (b) and have
$\alpha f(\alpha)'=0$ in $A$, and we must show that as well
$f(\alpha)'\alpha=0$. 
By the assumption, $f(\alpha)'\alpha = c \omega_i$ for some
$c\in K$. We have
$$c = \vf(c\omega_i) = \vf(f(\alpha)'\alpha) =  \vf(\alpha f(\alpha)') = \vf(0) = 0.
$$
Assume now that $r=4$.
Since $Q$ is $2$-regular, $Q$ is of the form
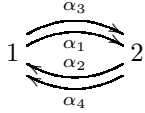
\[
   \xymatrix@=3pc{
    1 \ar@/^2ex/[r]_{\alpha_1}  \ar@/^2ex/@<+1ex>[r]^{\alpha_3} 
    & 2 \ar@/^2ex/[l]_{\alpha_2}  \ar@/^2ex/@<+1ex>[l]^{\alpha_4}
   }
\]
with $f$-orbit $(\alpha_1\ \alpha_2\ \alpha_3\ \alpha_4)$  
and $g$-orbit $(\alpha_1\ \alpha_4\ \alpha_3\ \alpha_2)$.
We adjust $\alpha_2, \alpha_3, \alpha_4$ as in (b) to have
$\alpha_1 \alpha_2 = 0$, $\alpha_2 \alpha_3 = 0$, $\alpha_3 \alpha_4 = 0$.
By assumption we have 
$\alpha_4 \alpha_1 = a B_{\alpha_4} = a (\alpha_4 \alpha_3 \alpha_2 \alpha_1)^m$
for some $a \in K$ and $m \in \mathbb{N}_{\geq 1}$.
We replace $\alpha_1$ by 
$\alpha'_1 = \alpha_1 - a \alpha_3 \alpha_2 \alpha_1 (\alpha_4 \alpha_3 \alpha_2 \alpha_1)^{m-1}$
and obtain $\alpha_4 \alpha'_1 = 0$.
Observe that we have also $\alpha'_1 \alpha_2 = 0$,
because $\alpha_1 \alpha_2 = 0$.

(III) \ We show  that relation (3) holds in $A$.
For each arrow $\alpha \in Q_1$, we have
$B_{\alpha} = c_{\alpha} \omega_{s(\alpha)}$
for some $c_{\alpha} \in K^*$.
We claim that $c_{\sigma} = c_{\alpha}$
for any arrow $\sigma$ in the $g$-orbit
$\cO(\alpha)$ of $\alpha$.
Indeed, if $\sigma$ belongs to $\cO(\alpha)$,
then
\begin{align*}
 c_{\sigma} 
  &= c_{\sigma} \varphi \big(\omega_{s(\sigma)})  
   = \varphi \big(c_{\sigma} \omega_{s(\sigma)})  
   = \varphi (B_{\sigma})  
   = \varphi (B_{\alpha})  
 \\ & 
   = \varphi \big(c_{\alpha} \omega_{s(\alpha)})  
   = c_{\alpha} \varphi \big(\omega_{s(\alpha)})  
   = c_{\alpha}.
\end{align*}
Since $K$ is algebraically closed, we may
choose $d_{\alpha} \in K^*$
such that
$d_{\alpha}^{m_{\alpha} n_{\alpha}} = c_{\alpha}^{-1}$.
Replacing now the representative
of each arrow $\alpha \in Q_1$ in $A$
by its product with $d_{\alpha}$,
we obtain a new presentation $A \cong K Q / I'$
such that $B_{\alpha} = \omega_{s(\alpha)}$
for any arrow $\alpha \in Q_1$.
This does not change the relations (1) obtained above.
Therefore, we may assume that,
if $i \in Q$ is any vertex,
$\alpha$ and $\bar{\alpha}$
are the arrows in $Q$ with source  $i$,
then
$B_{\alpha} = \omega_{i} = B_{\bar{\alpha}}$
in $A$.

(IV) \ We show that relation (2) holds in $A$.
When  the border $\partial(Q,f)$ of $(Q,f)$ is empty, there is
nothing to do (and $A$ is isomorphic to
$B$).
Assume now that $\partial(Q,f)$
is not empty.
Then for any loop $\alpha$ with
$i = s(\alpha) \in \partial(Q,f)$,
we have
$\alpha^2 = \alpha f(\alpha) = b_i \omega_i = b_i B_{\alpha}$
for some $b_i \in K$.
Hence, we have a border function
$b_{\bullet} : \partial(Q,f) \to K$,
and $A$ is isomorphic to the algebra
$B(Q, f, m_{\bullet}, b_{\bullet})$.
\end{proof}

Recall that a self-injective algebra $A$ is biserial if the
radical of any indecomposable non-uniserial projective, left or right,
$A$-module is a sum of two uniserial modules whose intersection is simple.

Theorem~\ref{th:main3} follows from
Theorems \ref{th:2.6}, \ref{th:3.1}, \ref{th:5.3} 
and the following relative version of 
Theorem~\ref{th:4.1} 
(see Remark~\ref{rem:4.3}). 

\begin{theorem}
\label{th:5.4}  
Let $B= B(Q, f, m_{\bullet})$ where $Q$ has 
at least two vertices, and where the border  $\partial(Q,f)$ is not empty.
Then there is a  canonically defined
weighted triangulation quiver $(Q^\#,f^\#,m_{\bullet}^\#)$ 
such that the following statements hold.
\begin{enumerate}[(i)]
  \item
    $|\partial(Q,f)| = |\partial(Q^\#,f^\#)|$.
  \item
    $B$ is isomorphic to the idempotent algebra $e^\# B^\# e^\#$
    of the biserial weighted triangulation algebra
    $B^\# = B(Q^\#,f^\#,m_{\bullet}^\#)$
    with respect to a canonically defined idempotent $e^\#$ of $B^\#$.    
  \item
    For any border function $b_{\bullet}$ of $(Q,f)$
    and the induced border function $b_{\bullet}^\#$ of $(Q^\#,f^\#)$,
    the algebras
    $B(Q,f,m_{\bullet},b_{\bullet})$
    and
    $e^\# B(Q^\#,f^\#,m_{\bullet}^\#,b_{\bullet}^\#) e^\#$
    are isomorphic.
\end{enumerate}
\end{theorem}

\begin{proof}  
The construction of $(Q^{\#}, f^{\#}, m_{\bullet}^{\#})$
is analogous to the $*$-construction in
Theorem \ref{th:4.1}. 
We take the notation as in Theorem \ref{th:4.1} 
and in addition we denote by $Q_1^b$ the set of all border loops 
of the quiver. 
We define a triangulation quiver $(Q^\#, f^\#)$ as follows.
We take $Q^\# = (Q^\#_0,Q^\#_1,s^\#,t^\#)$ with
\begin{align*}
  Q_0^\# & := Q_0 \cup \{x_{\alpha}\}_{\alpha \in Q_1 \setminus Q_1^b} , &
  Q_1^\# & := Q_1^b \cup 
      \{\alpha', \alpha'', \varepsilon_{\alpha} \}_{\alpha \in Q_1 \setminus Q_1^b}, 
\end{align*}
$s^\#(\beta) = s(\beta) = t(\beta) = t^\#(\beta)$
for all loops $\beta \in Q_1^b$,
and 
$s^\#(\alpha') = s(\alpha)$,
$t^\#(\alpha') = x_{\alpha}$,
$s^\#(\alpha'') = x_{\alpha}$,
$t^\#(\alpha'') = t(\alpha)$,
$s^\#(\varepsilon_{\alpha}) = x_{f(\alpha)}$,
$t^\#(\varepsilon_{\alpha}) = x_{\alpha}$,
for any arrow $\alpha \in Q_1 \setminus Q_1^b$.
Moreover, we set 
$f^\#(\eta) = \eta$ for any loop $\eta \in Q_1^b$,
and
$f^\#(\alpha'') = f(\alpha)'$,
$f^\#(f(\alpha)') = \varepsilon_{\alpha}$,
$f^\#(\varepsilon_{\alpha}) = \alpha''$,
for any arrow $\alpha \in Q_1 \setminus Q_1^b$.
We observe that $(Q^\#, f^\#)$ is 
a triangulation quiver
with $\partial(Q^\#,f^\#) = \partial(Q,f)$.
Let $g^\#$ be the permutation of $Q_1^\#$
associated to $f^\#$.
For each arrow $\beta$ in $Q_1^\#$,
we denote by $\cO^\#(\beta)$ the $g^\#$-orbit
of $\beta$ in $Q_1^\#$.
Then the  $g^\#$-orbits in $Q_1^\#$
are
\begin{align*}
  \cO^\#(\eta) & = \big(\eta\ g(\eta)' \ g(\eta)'' \ \dots 
                                \ g^{n_{\eta}-1}(\eta)' \ g^{n_{\eta}-1}(\eta)''\big) , 
\end{align*}
for any loop $\eta \in Q_1^b$, and
\begin{align*}
  \cO^\#(\alpha') & = \big(\alpha' \ \alpha'' \ g(\alpha)' \ g(\alpha)'' \ \dots 
                                \ g^{n_{\alpha}-1}(\alpha)' \ g^{n_{\alpha}-1}(\alpha)''\big) , \\
  \cO^\#(\varepsilon_{\alpha}) & = \big(\varepsilon_{f^{r_{\alpha}-1}(\alpha)} 
                                \ \varepsilon_{f^{r_{\alpha}-2}(\alpha)} \ \dots 
                                \ \varepsilon_{f(\alpha)} \ \varepsilon_{\alpha}\big) , 
\end{align*}
for any arrow $\alpha \in Q_1 \setminus Q_1^b$ 
(where $r_{\alpha}$ is the length of the $f$-orbit of $\alpha$).
We define the weight function 
${m}_{\bullet}^\# : \cO(g^\#) \to \bN^*$ 
by 
${m}_{\cO^\#(\eta)}^\# = m_{\eta}$ 
for any loop $\eta \in Q_1^b$, and
${m}_{\cO^\#(\alpha')}^\# = m_{\alpha}$ 
and
${m}_{\cO^\#(\varepsilon_{\alpha})}^\# = 1$
for any arrow $\alpha \in Q_1 \setminus Q_1^b$.

Let $B^\# = B(Q^\#,f^\#,m_{\bullet}^\#)$
be the biserial weighted triangulation algebra
associated to $(Q^\#,f^\#,m_{\bullet}^\#)$
and  $e^\#$ the sum of the primitive idempotents
$e_i^\#$ in $B^\#$ associated to the vertices $i \in Q_0$.   
Then it follows from the arguments 
as in the proof of Proposition~\ref{prop:2.7}
that $B$ is isomorphic to the idempotent algebra $e^\# B^\# e^\#$.
Moreover, let $b_{\bullet}$ be a border function of $(Q,f)$ 
and
$b_{\bullet}^\#$ be the induced border function of $(Q^\#,f^\#)$,
that is $b_{i}^\# = b_{i}$ for any border vertex $i$.
Then it follows from the description of 
$g^\#$-orbits in $Q_1^\#$
and the definition of the weight function 
${m}_{\bullet}^\#$ that
$B(Q,f,m_{\bullet},b_{\bullet})$
is isomorphic to the idempotent algebra
$e^\# B(Q^\#,f^\#,m_{\bullet}^\#,b_{\bullet}^\#) e^\#$.
\end{proof} 

\begin{example}
This illustrates the $\#$-construction in 
Theorem \ref{th:5.4}. Let $(Q,f)$ be the biserial quiver 
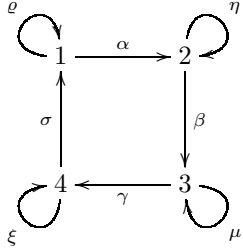
\[
  \xymatrix@C=3pc@R=3pc{
    1
      \ar[r]^{\alpha}
      \ar@(l,u)^{\varrho}[] 
  &
    2
      \ar[d]^{\beta}
      \ar@(u,r)^{\eta}[] 
  \\ 
    4
      \ar[u]^{\sigma}
      \ar@(d,l)^{\xi}[] 
  &
    3
      \ar[l]^{\gamma}
      \ar@(r,d)^{\mu}[] 
  } 
\]
with  $f$-orbits
$(\alpha \ \beta \ \gamma\ \sigma)$,
$(\varrho)$,
$(\eta)$,
$(\mu)$,
$(\xi)$.
Then the border $\partial(Q,f)$ of $(Q,f)$ is the set
$Q_0 = \{1,2,3,4\}$ of all vertices of $Q$, and
$\varrho, \eta, \mu, \xi$ are the border loops.
Further, 
$g$ has only one orbit,  
$\cO(\alpha) = (\alpha \ \eta\ \beta \ \mu\ \gamma\ \xi\ \sigma\ \varrho)$.
We take the weight function
$m_{\bullet} : \cO(g) \to \bN^*$
with $m_{\cO(\alpha)} = 1$.
Moreover, let $b_{\bullet} : \partial(Q,f) \to K$
be a border function. 
Then we describe the associated algebra 
$B(Q,f,m_{\bullet},b_{\bullet})$.
It has quiver $Q$, and to simplify the notation for
the relations, we use the notion of $B_{\alpha}$ for an 
arrow $\alpha$, as it has appeared throughout,
\begin{align*}
 \varrho^2 &= b_1B_{\rho}, &  \ B_{\rho} &= B_{\alpha}, & \ \alpha\beta &=0,
\\
 \eta^2 &= b_2 B_{\eta}, &  \ B_{\eta} &= B_{\beta}, & \ 
 \beta \gamma &= 0,
\\
 \mu^2 &= b_3 B_{\mu}, & \ 
B_{\mu} 
     &= B_{\gamma}, & \ 
 \gamma \sigma &= 0,
\\
 \xi^2 &= b_4 B_{\xi}, & \ 
 B_{\xi}  &= B_{\sigma}, & \ 
 \sigma \alpha &= 0.
\end{align*}
Note that the algebra $B(Q,f,m_{\bullet})$
is given by the quiver $Q$ and the above relations
such that all $b_i$ are zero.
By the arguments as in
\cite[Example~6.5]{ESk4}, if $K$ has characteristic $2$ and 
$b_{\bullet}$ is non-zero, 
then the algebras 
$B(Q,f,m_{\bullet},b_{\bullet})$
and
$B(Q,f,m_{\bullet})$
are not isomorphic.

The triangulation quiver
$(Q^\#,f^\#)$ is of the form
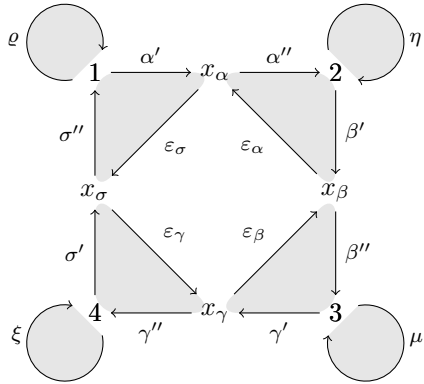
\[
\begin{tikzpicture}
[->,scale=.8]
\coordinate (1) at (-2,2) ;
\coordinate (2) at (2,2) ;
\coordinate (3) at (2,-2) ;
\coordinate (4) at (-2,-2) ;
\coordinate (up) at (0,2) ;
\coordinate (down) at (0,-2) ;
\coordinate (right) at (2,0) ;
\coordinate (left) at (-2,0) ;
\fill[rounded corners=3mm,fill=gray!20] (1) -- (up) -- (left) -- cycle;
\fill[rounded corners=3mm,fill=gray!20] (2) -- (right) -- (up) -- cycle;
\fill[rounded corners=3mm,fill=gray!20] (3) -- (down) -- (right) -- cycle;
\fill[rounded corners=3mm,fill=gray!20] (4) -- (left) -- (down) -- cycle;
\node (1) at (-2,2) {$1$};
\node (2) at (2,2) {$2$};
\node (3) at (2,-2) {$3$};
\node (4) at (-2,-2) {$4$};
\node (up) at (0,2) {$x_{\alpha}$};
\node (down) at (0,-2) {$x_{\gamma}$};
\node (right) at (2,0) {$x_{\beta}$};
\node (left) at (-2,0) {$x_{\sigma}$};
\node [circle,minimum size=.5cm](A) at (-2,2) {1};
\node [circle,minimum size=1.cm](B) at (-2.5,2.5) {};
\coordinate  (C1) at (intersection 2 of A and B);
\coordinate  (D1) at (intersection 1 of A and B);
 \tikzAngleOfLine(B)(D1){\AngleStart}
 \tikzAngleOfLine(B)(C1){\AngleEnd}
\fill[gray!20]%
   let \p1 = ($ (B) - (D1) $), \n2 = {veclen(\x1,\y1)}
   in
     (D1) arc (\AngleStart-360:\AngleEnd:\n2); 
\draw[<-]%
   let \p1 = ($ (B) - (D1) $), \n2 = {veclen(\x1,\y1)}
   in
     (B) ++(60:\n2) node[below]{\footnotesize\raisebox{-2.5ex}{$\varrho\qquad\qquad\qquad\ $}}
     (D1) arc (\AngleStart-360:\AngleEnd:\n2); 
\node [circle,minimum size=.5cm](A) at (2,2) {2};
\node [circle,minimum size=1.cm](B) at (2.5,2.5) {};
\coordinate  (C1) at (intersection 2 of A and B);
\coordinate  (D1) at (intersection 1 of A and B);
 \tikzAngleOfLine(B)(D1){\AngleStart}
 \tikzAngleOfLine(B)(C1){\AngleEnd}
\fill[gray!20]%
   let \p1 = ($ (B) - (D1) $), \n2 = {veclen(\x1,\y1)}
   in
     (D1) arc (\AngleStart-360:\AngleEnd:\n2); 
\draw[<-]%
   let \p1 = ($ (B) - (D1) $), \n2 = {veclen(\x1,\y1)}
   in
     (B) ++(60:\n2) node[below right]{\footnotesize\raisebox{-2.5ex}{$\ \ \eta$}}
     (D1) arc (\AngleStart-360:\AngleEnd:\n2); 
\node [circle,minimum size=.5cm](A) at (2,-2) {3};
\node [circle,minimum size=1.cm](B) at (2.5,-2.5) {};
\coordinate  (C1) at (intersection 2 of A and B);
\coordinate  (D1) at (intersection 1 of A and B);
 \tikzAngleOfLine(B)(D1){\AngleStart}
 \tikzAngleOfLine(B)(C1){\AngleEnd}
\fill[gray!20]%
   let \p1 = ($ (B) - (D1) $), \n2 = {veclen(\x1,\y1)}
   in
     (D1) arc (\AngleStart-360:\AngleEnd:\n2); 
\draw[<-]%
   let \p1 = ($ (B) - (D1) $), \n2 = {veclen(\x1,\y1)}
   in
     (B) ++(60:\n2) node[below right]{\footnotesize\raisebox{-2.5ex}{$\ \ \mu$}}
     (D1) arc (\AngleStart-360:\AngleEnd:\n2); 
\node [circle,minimum size=.5cm](A) at (-2,-2) {4};
\node [circle,minimum size=1.cm](B) at (-2.5,-2.5) {};
\coordinate  (C) at (intersection 2 of A and B);
\coordinate  (D) at (intersection 1 of A and B);
 \tikzAngleOfLine(B)(D){\AngleStart}
 \tikzAngleOfLine(B)(C){\AngleEnd}
\fill[gray!20]%
   let \p1 = ($ (B) - (D) $), \n2 = {veclen(\x1,\y1)}
   in
     (D) arc (\AngleStart-360:\AngleEnd:\n2); 
\draw[<-]%
   let \p1 = ($ (B) - (D) $), \n2 = {veclen(\x1,\y1)}
   in
     (B) ++(60:\n2) node[left]{\footnotesize\raisebox{-7ex}{$\xi\qquad\ $}}
     (D) arc (\AngleStart-360:\AngleEnd:\n2); 
\draw
(1) edge node[above]{\footnotesize$\alpha'$} (up)
(up) edge node[above]{\footnotesize$\alpha''$} (2)
(2) edge node[right]{\footnotesize$\beta'$} (right)
(right) edge node[right]{\footnotesize$\beta''$} (3)
(3) edge node[below]{\footnotesize$\gamma'$} (down)
(down) edge node[below]{\footnotesize$\gamma''$} (4)
(4) edge node[left]{\footnotesize$\sigma'$} (left)
(left) edge node[left]{\footnotesize$\sigma''$} (1)
(right) edge node[below left]{\footnotesize$\varepsilon_{\alpha}$} (up)
(down) edge node[above left]{\footnotesize$\varepsilon_{\beta}$} (right)
(left) edge node[above right]{\footnotesize$\varepsilon_{\gamma}$} (down)
(up) edge node[below right]{\footnotesize$\varepsilon_{\sigma}$} (left)
;
\end{tikzpicture}
\]
with  $f^\#$-orbits
$(\varrho)$,
$(\eta)$,
$(\mu)$,
$(\xi)$,
$(\alpha' \ \varepsilon_{\sigma} \ \sigma'')$,
$(\beta' \ \varepsilon_{\alpha} \ \alpha'')$,
$(\gamma' \ \varepsilon_{\beta} \ \beta'')$,
$(\sigma' \ \varepsilon_{\gamma} \ \gamma'')$.
Further, 
there are  two $g^\#$-orbits:
\begin{align*}
  \cO^\#(\alpha')  &= (\alpha' \ \alpha'' \ \eta\ \beta' \ \beta''\ \mu
                                \ \gamma' \ \gamma'' \ \xi\ \sigma' \ \sigma''\ \varrho) , 
                                \\
  \cO^\#(\varepsilon_{\alpha})  &= (\varepsilon_{\alpha}
                                \ \varepsilon_{\sigma} 
                                \ \varepsilon_{\gamma} 
                                \ \varepsilon_{\beta} 
                                ) .
\end{align*}
The weight function 
$m_{\bullet}^\#$
takes only value $1$,  
and the border function 
$b_{\bullet}^\#$
is 
$b_1^\# = b_1$,
$b_2^\# = b_2$,
$b_3^\# = b_3$,
$b_4^\# = b_4$.

(a) The relations from vertex 1 are
$$\rho^2= b_1B_{\rho}, \ \ \ B_{\rho} = B_{\alpha'}.
$$ There are analogous
relations from each of the vertices $2, 3, 4$.

(b) The relations from vertex $x_{\alpha}$ are 
$$B_{\varepsilon_{\sigma}} = B_{\alpha''}, \ \ \alpha''\beta'=0, \ \ \varepsilon_{\sigma}\sigma'' = 0.
$$
There are analogous relations from each of the vertices
$x_{\beta}, x_{\gamma}, x_{\sigma}$.

We observe now that 
$B(Q,f,m_{\bullet},b_{\bullet})$
is isomorphic to the idempotent algebra
$e^\# B(Q^\#,f^\#,m_{\bullet}^\#,b_{\bullet}^\#) e^\#$
where the idempotent $e^\#$ is
the sum of the primitive idempotents
at the vertices $1,2,3,4$.
Moreover, the algebras
$B(Q,f,m_{\bullet})$
and
$e^\# B(Q^\#,f^\#,m_{\bullet}^\#) e^\#$
are also isomorphic.
Finally, we note
that if $K$ has  characteristic $2$ and 
$b_{\bullet}^\# = b_{\bullet}$ is non-zero, 
then the algebras 
$B(Q^\#,f^\#,m_{\bullet}^\#,b_{\bullet}^\#)$
and
$B(Q^\#,f^\#,m_{\bullet}^\#)$
are not isomorphic.
\end{example}

\section{Proof of Theorem \ref{th:main4}}\label{sec:proof4}

We recall the definition of a weighted triangulation algebra. 
Let $(Q,f)$ be a triangulation quiver
with at least two vertices, and let $g$, $n_{\bullet}$ and $m_{\bullet}$ 
be defined as for
biserial quiver algebras.
The additional datum is a function
\[
  c_{\bullet} : \cO(g) \to K^* = K \setminus \{0\}
\]
which we call 
a \emph{parameter function} of $(Q,f)$.
We write briefly $m_{\alpha} = m_{\cO(\alpha)}$
and $c_{\alpha} = c_{\cO(\alpha)}$ for $\alpha \in Q_1$.
The parameter function $c_{\bullet}$ taking only value $1$
is said to be \emph{trivial}.
\emph{We assume that $m_{\alpha} n_{\alpha} \geq 3$
for any arrow $\alpha \in Q_1$.}
For any arrow $\alpha \in Q_1$,  define the path
\begin{align*}
  A_{\alpha} &= \Big( \alpha g(\alpha) \dots g^{n_{\alpha}-1}(\alpha)\Big)^{m_{\alpha}-1}
             \alpha g(\alpha) \dots g^{n_{\alpha}-2}(\alpha) , \mbox{ if } n_{\alpha} \geq 2, \\
  A_{\alpha} &= \alpha^{m_{\alpha}-1} , \mbox{ if } n_{\alpha} = 1 ,
\end{align*}
in $Q$ of length $m_{\alpha} n_{\alpha} - 1$. Then we have
\[
  A_{\alpha} g^{n_{\alpha}-1}(\alpha) =  B_{\alpha} = \Big( \alpha g(\alpha) \dots g^{n_{\alpha}-1}(\alpha)\Big)^{m_{\alpha}}
\]
of length $m_{\alpha} n_{\alpha}$.
Then, following \cite{ESk3},
we define  the bound quiver algebra
\[
  \Lambda(Q,f,m_{\bullet},c_{\bullet})
   = K Q / I (Q,f,m_{\bullet},c_{\bullet}),
\]
where $I (Q,f,m_{\bullet},c_{\bullet})$
is the admissible ideal in the path algebra $KQ$ of $Q$ over $K$
generated by the elements:
\begin{enumerate}[(1)]
 \item
  ${\alpha} f({\alpha}) - c_{\bar{\alpha}} A_{\bar{\alpha}}$,
  for all arrows $\alpha \in Q_1$,
 \item
  $\beta f(\beta) g(f(\beta))$,
  for all arrows $\beta \in Q_1$.
\end{enumerate}
The algebra $\Lambda:= \Lambda(Q,f,m_{\bullet},c_{\bullet})$ is called a
\emph{weighted triangulation algebra} of $(Q,f)$.
Moreover, if 
$(Q,f) = (Q(S,\vv{T}),f)$  
for a directed triangulated surface $(S,\vv{T})$,
then
$\Lambda$
is called a \emph{weighted surface algebra},
and if the surface and triangulation is important we denote 
the algebra by $\Lambda(S,\vv{T},m_{\bullet},c_{\bullet})$.

We note that the Gabriel quiver of $\Lambda$ is equal to  $Q$, 
this holds because
we assume 
$m_{\alpha} n_{\alpha} \geq 3$
for all arrows $\alpha \in Q_1$.

We have the following proposition
(see \cite[Proposition~5.8]{ESk3}).

\begin{proposition}
\label{prop:6.1}
Let $(Q,f)$ be a triangulation quiver,
$m_{\bullet}$ and $c_{\bullet}$
weight and parameter functions of $(Q,f)$.
Then $\Lambda = \Lambda(Q,f,m_{\bullet},c_{\bullet})$
is a finite-dimensional 
tame symmetric
algebra of dimension
$\sum_{\cO \in \cO(g)} m_{\cO} n_{\cO}^2$.
\end{proposition}

We have also the following theorem
proved in \cite[Theorem~1.2]{ESk3}
(see also \cite[Proposition~7.1]{BES} and \cite[Theorem~5.9]{ESk2}
for the case of two vertices).

\begin{theorem}
\label{th:6.2}
Let $\Lambda = \Lambda(S,\vv{T},m_{\bullet},c_{\bullet})$
be a weighted surface algebra over an algebraically
closed field $K$.
Then the following statements are equivalent:
\begin{enumerate}[(i)]
 \item
  All simple modules in $\mod \Lambda$ are periodic of period $4$.
 \item
  $\Lambda$ is a periodic algebra of period $4$.
 \item
  $\Lambda$ is not isomorphic to a singular tetrahedral algebra.
\end{enumerate}
\end{theorem}

Following \cite{ESk3},
a singular  tetrahedral algebra is 
the weighted surface algebra
given by a coherent orientation of 
four triangles of the tetrahedron and the 
weight and parameter functions taking only value $1$.
The triangulation quiver of such algebra  is the 
\emph{tetrahedral quiver} of the form
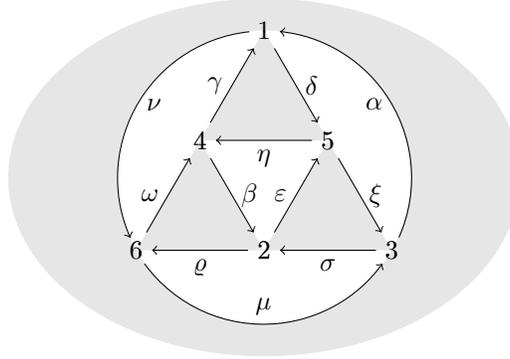
\[
\begin{tikzpicture}
[->,scale=.85]
\node (1) at (0,1.72) {$1$};
\node (2) at (0,-1.72) {$2$};
\node (3) at (2,-1.72) {$3$};
\node (4) at (-1,0) {$4$};
\node (5) at (1,0) {$5$};
\node (6) at (-2,-1.72) {$6$};
\coordinate (1) at (0,1.72);
\coordinate (2) at (0,-1.72);
\coordinate (3) at (2,-1.72);
\coordinate (4) at (-1,0);
\coordinate (5) at (1,0);
\coordinate (6) at (-2,-1.72);
\fill[fill=gray!20]
      (0,2.22cm) arc [start angle=90, delta angle=-360, x radius=4cm, y radius=2.8cm]
 --  (0,1.72cm) arc [start angle=90, delta angle=360, radius=2.3cm]
     -- cycle;
\fill[fill=gray!20]
    (1) -- (4) -- (5) -- cycle;
\fill[fill=gray!20]
    (2) -- (4) -- (6) -- cycle;
\fill[fill=gray!20]
    (2) -- (3) -- (5) -- cycle;

\node (1) at (0,1.72) {$1$};
\node (2) at (0,-1.72) {$2$};
\node (3) at (2,-1.72) {$3$};
\node (4) at (-1,0) {$4$};
\node (5) at (1,0) {$5$};
\node (6) at (-2,-1.72) {$6$};
\draw (-.23,1.7) arc [start angle=96, delta angle=108, radius=2.3cm] node[midway,right] {$\nu$};
\draw (-1.87,-1.93) arc [start angle=-144, delta angle=108, radius=2.3cm] node[midway,above] {$\mu$};
\draw (2.11,-1.52) arc [start angle=-24, delta angle=108, radius=2.3cm] node[midway,left] {$\alpha$};
\draw
(1) edge node [right] {$\delta$} (5)
(2) edge node [left] {$\varepsilon$} (5)
(2) edge node [below] {$\varrho$} (6)
(3) edge node [below] {$\sigma$} (2)
(4) edge node [left] {$\gamma$} (1)
(4) edge node [right] {$\beta$} (2)
(5) edge node [right] {$\xi$} (3)
(5) edge node [below] {$\eta$} (4)
(6) edge node [left] {$\omega$} (4)
;
\end{tikzpicture}
\]
where
the shaded triangles denote $f$-orbits 
and white triangles denote $g$-orbits.

The following theorem is an essential ingredient
for the proof of Theorem~\ref{th:main4}.

\begin{theorem}
\label{th:6.3}
Let 
$B = B(Q,f,m_{\bullet})$ be a biserial weighted triangulation algebra
where  $Q$ has no loops, and
$\Lambda^* = \Lambda(Q^*,f^*,m^*_{\bullet},c^*_{\bullet})$
the weighted triangulation algebra 
associated to the weighted triangulation quiver
$(Q^*,f^*,m^*_{\bullet})$ and
the trivial parameter function  $c^*_{\bullet}$
of $(Q^*,f^*)$.
Then the following statements hold:
\begin{enumerate}[(i)]
 \item
  $\Lambda^*$ is a periodic algebra of period $4$.
  \item
    $B$ is isomorphic to the idempotent algebra $e^{*} \Lambda^{*} e^{*}$
    for an idempotent $e^{*}$ of $\Lambda^{*}$.    
\end{enumerate}
\end{theorem}

\begin{proof}
For each arrow $\varrho$ in $Q^*_1$, we set
$m^*_{\varrho} = m^*_{\cO(\varrho)}$ and
$n^*_{\varrho} = |\cO^*(\varrho)|$.
We observe first that $n^*_{\varrho} \geq 3$, and hence
$m^*_{\varrho} n^*_{\varrho} \geq 3$,
for any arrow $\varrho$ in $Q^*_1$,
and consequently 
$\Lambda(Q^*,f^*,m^*_{\bullet},c^*_{\bullet})$
is a well defined weighted triangulation algebra.
Indeed, it follows from 
Theorem~\ref{th:4.1}(iii)
that 
the triangulation quiver $(Q^*,f^*)$
has neither loops nor  self-folded triangles.
Moreover, the $f$-orbits in $Q_1$ have length $3$, 
and the $g$-orbits in $Q_1$ are of length at least $2$.
Then it follows from the proof of Theorem~\ref{th:4.1}
that the
 $g^*$-orbits in $Q_1^*$ are
\begin{align*}
  \cO^*(\alpha') & = \big(\alpha' \ \alpha'' \ g(\alpha)' \ g(\alpha)'' \ \dots 
                                \ g^{n_{\alpha}-1}(\alpha)' \ g^{n_{\alpha}-1}(\alpha)''\big) , \\
  \cO^*(\varepsilon_{\alpha}) & = \big(\varepsilon_{f^{2}(\alpha)} 
                                \ \varepsilon_{f(\alpha)} \ \varepsilon_{\alpha}\big) , 
\end{align*}
for all arrows $\alpha \in Q_1$.
Then the required inequalities hold.
Further, it follows from Remark~\ref{rem:4.4} that $(Q^*,f^*)$
is not the tetrahedral quiver.
Then, applying Theorem~\ref{th:6.2}, we conclude that
$\Lambda^*$ is a periodic algebra of period $4$.

Let $e^{*}$ be the sum of all primitive idempotents in $\Lambda^{*}$
corresponding to the vertices of $Q$.
We claim that $e^{*} \Lambda^{*} e^{*}$ is isomorphic to $B$.
Observe that every $f$-orbit $(\alpha \ \beta \ \gamma)$ in $Q_1$
%
creates in $(Q^*,f^*)$ the subquiver
as in the illustration (3) following Theorem \ref{th:4.1}.
%
%
The algebra
$e^{*} \Lambda^{*} e^{*}$ has arrows 
$\alpha = \alpha' \alpha''$,
$\beta = \beta' \beta''$,
$\gamma = \gamma' \gamma''$,
and it follows that in $e^{*}\Lambda^{*}e^{*}$ we have
\begin{align*}
 \alpha \beta & = \alpha' \alpha'' \beta' \beta'' 
   = \alpha' \alpha'' f^*(\alpha'') g^*\big(f^*(\alpha'')\big) = 0, \\
 \beta \gamma & = \beta' \beta''  \gamma' \gamma''
   = \beta' \beta'' f^*(\beta'') g^*\big(f^*(\beta'')\big) = 0, \\
 \gamma \alpha & = \gamma' \gamma'' \alpha' \alpha'' 
   = \gamma' \gamma'' f^*(\gamma'') g^*\big(f^*(\gamma'')\big) = 0 .
\end{align*}
Further, 
let $i$ be a vertex of $Q$,  and let  $\alpha$ and $\sigma = \bar{\alpha}$
the two arrows in $Q_1$ with source $i$.
By the proof of Theorem~\ref{th:4.1}, the 
$g^*$-orbits are
\begin{align*}
  \cO^*(\alpha') & = \big(\alpha' \ \alpha'' \ g(\alpha)' \ g(\alpha)'' \ \dots 
                                \ g^{n_{\alpha}-1}(\alpha)' \ g^{n_{\alpha}-1}(\alpha)''\big) , \\
  \cO^*(\sigma') & = \big(\sigma' \ \sigma'' \ g(\sigma)' \ g(\sigma)'' \ \dots 
                                \ g^{n_{\sigma}-1}(\sigma)' \ g^{n_{\sigma}-1}(\sigma)''\big) .
\end{align*}
Moreover, 
$m_{\alpha'}^* = m_{\cO^*(\alpha')}^* = m_{\cO(\alpha)} = m_{\alpha}$
and
$m_{\sigma'}^* = m_{\cO^*(\sigma')}^* = m_{\cO(\sigma)} = m_{\sigma}$.
Hence we have in $Q^*$ the cycles 
\begin{align*}
  B_{\alpha'} & = \big(\alpha' \alpha'' g(\alpha)' g(\alpha)'' \dots 
                                g^{n_{\alpha}-1}(\alpha)' g^{n_{\alpha}-1}(\alpha)''\big)^{m_{\alpha}} , \\
  B_{\sigma'} & = \big(\sigma' \sigma'' g(\sigma)' \ g(\sigma)'' \dots 
                                g^{n_{\sigma}-1}(\sigma)' g^{n_{\sigma}-1}(\sigma)''\big)^{m_{\sigma}} ,
\end{align*}
and $B_{\alpha'}  = B_{\sigma'}$ in $\Lambda^*$
(see \cite[Lemma~5.3]{ESk3}),
and this gives the equality 
$B_{\alpha}  = B_{\sigma} = B_{\bar{\alpha}}$
in $e^{*} \Lambda^{*} e^{*}$.
Therefore, $e^{*} \Lambda^{*} e^{*}$ is isomorphic to $B$.
\end{proof}

We may now complete the proof of Theorem~\ref{th:main4}.
Let $B = B(Q,f,m_{\bullet})$ be a biserial quiver algebra.
Then it follows from Theorem~\ref{th:4.1} that
$B$ is isomorphic to the idempotent algebra $e^{**}B^{**}e^{**}$
of the biserial triangulation algebra
$B^{**} = B(Q^{**},f^{**},m_{\bullet}^{**})$ for some idempotent
$e^{**}$ of $B^{*}$, and $Q^{**}$ has no loops.
Applying now Theorem~\ref{th:6.3} we conclude that
$B^{**}$ is isomorphic to the idempotent algebra
$e \Lambda e$ of a periodic weighted triangulation algebra,
for an idempotent $e$ of $\Lambda$.
Since $e^{**}$ is a summand of $e$, we have
$B \cong e^{**}B^{**}e^{**} \cong e^{**}(e \Lambda e)e^{**} = e^{**} \Lambda e^{**}$.
Then Theorem~\ref{th:main4} follows from 
Theorems \ref{th:2.6} and \ref{th:3.1}.

\begin{remark}
\label{rem:6.4}
Let $\Lambda = \Lambda(Q,f,m_{\bullet},c_{\bullet})$
be a weighted triangulation algebra.
Then the biserial  triangulation algebra
$B = B(Q,f,m_{\bullet})$
is not an idempotent algebra $e \Lambda e$ of $\Lambda$.
On the other hand, 
if $\Lambda$ is not a tetrahedral algebra, then $B$ is
a geometric degeneration of $\Lambda$
(see \cite[Proposition~5.8]{ESk3}).
\end{remark}

\begin{example}
\label{ex:6.5}
Let $(Q,f)$ be the Markov quiver in Example \ref{ex:3.5}
and $m$ a positive integer associated
to the unique $g$-orbit
$(\alpha_1\, \beta_2\, \alpha_3\, \beta_1\, \alpha_2\, \beta_3)$ in $Q_1$.
Then the
associated weighted triangulation algebra 
$\Lambda = \Lambda(Q,f,m_{\bullet},c_{\bullet})$
with  trivial  parameter function  $c^*_{\bullet}$
is given by the quiver $Q$ and the following relations 
(we write the indices modulo 3):
\begin{align*}
 \alpha_i \alpha_{i+1}  &= (\beta_i \alpha_{i+1} \beta_{i+2}\alpha_i \beta_{i+1} \alpha_{i+2})^{m-1}  \beta_i \alpha_{i+1} \beta_{i+2} \alpha_i\beta_{i+1},
 &
 \alpha_i \alpha_{i+1} \beta_{i+2} &= 0,
\\
 \beta_i \beta_{i+1} &= (\alpha_i \beta_{i+1} \alpha_{i+2} \beta_i \alpha_{i+1} \beta_{i+2})^{m-1}  \alpha_i \beta_{i+1} \alpha_{i+2} \beta_i \alpha_{i+1},
 &
 \beta_i \beta_{i+1} \alpha_{i+2} &= 0.
\end{align*}
The  idempotent algebra $e_1 \Lambda e_1$
of $\Lambda$ with respect to the primitive idempotent $e_1$
at  vertex $1$ is isomorphic to the Brauer graph algebra
$B_{\Gamma}$ given by the Brauer graph $\Gamma$ in Example \ref{ex:4.5}. 

According to Theorem~\ref{th:4.1}
we have the triangulation quiver $(Q^*,f^*)$
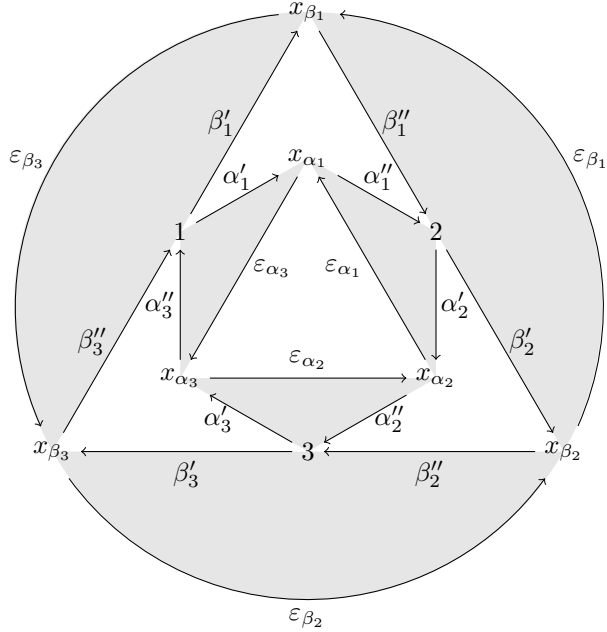
\[
\begin{tikzpicture}
[->,scale=.85]
\coordinate (1) at (0,2.29);
\coordinate (3) at (2,-1.15);
\coordinate (6) at (-2,-1.15);
\coordinate (1b) at (0,4.58);
\coordinate (2b) at (0,-2.3);
\coordinate (3b) at (4,-2.3);
\coordinate (4b) at (-2,1.14);
\coordinate (5b) at (2,1.14);
\coordinate (6b) at (-4,-2.3);
\fill[fill=gray!20]
      (1b) arc [start angle=90, delta angle=-360, radius=4.6cm]
 --  (6b) -- (3b)  -- cycle;
\fill[fill=gray!20]
    (1) -- (5b) -- (3) -- cycle;
\fill[fill=gray!20]
    (3) -- (2b) -- (6) -- cycle;
\fill[fill=gray!20]
    (6) -- (4b) -- (1) -- cycle;

\node (1) at (0,2.29) {$x_{\alpha_1}$};
\node (3) at (2,-1.15) {$x_{\alpha_2}$};
\node (6) at (-2,-1.15) {$x_{\alpha_3}$};
\node (1b) at (0,4.58) {$x_{\beta_1}$};
\node (2b) at (0,-2.3) {$3$};
\node (3b) at (4,-2.3) {$x_{\beta_2}$};
\node (4b) at (-2,1.14) {$1$};
\node (5b) at (2,1.14) {$2$};
\node (6b) at (-4,-2.3) {$x_{\beta_3}$};
\draw (-.46,4.54) arc [start angle=96, delta angle=108, radius=4.6cm] node[midway,left] {$\varepsilon_{\beta_3}$};
\draw (-3.74,-2.72) arc [start angle=-144, delta angle=108, radius=4.6cm] node[midway,below] {$\varepsilon_{\beta_2}$};
\draw (4.22,-1.9) arc [start angle=-24, delta angle=108, radius=4.6cm] node[midway,right] {$\varepsilon_{\beta_1}$};
\draw
(1b) edge node [right] {$\beta_1''$} (5b)
(2b) edge node [below] {$\beta_3'$} (6b)
(3b) edge node [below] {$\beta_2''$} (2b)
(4b) edge node [left] {$\beta_1'$} (1b)
(5b) edge node [right] {$\beta_2'$} (3b)
(6b) edge node [left] {$\beta_3''$} (4b)
(1) edge node [right] {$\varepsilon_{\alpha_3}$} (6)
(6) edge node [above] {$\varepsilon_{\alpha_2}$} (3)
(3) edge node [left] {$\varepsilon_{\alpha_1}$} (1)
(4b) edge node [above] {$\alpha_1'$} (1)
(1) edge node [above] {$\alpha_1''$} (5b)
(5b) edge node [right] {$\!\alpha_2'$} (3)
(3) edge node [right] {$\alpha_2''$} (2b)
(2b) edge node [left] {$\!\!\alpha_3'\ $} (6)
(6) edge node [left] {$\alpha_3''\!$} (4b)
;
\end{tikzpicture}
\]
where the shaded triangles denote the $f^*$-orbits in $Q_1^*$.
The $g^*$-orbits in $Q^*_1$ are
\begin{gather*}
  \cO^*(\alpha_1')  = (\alpha_1' \ \alpha_1'' \ \beta_2' \ \beta_2''
                                \ \alpha_3' \ \alpha_3'' \ \beta_1' \ \beta_1''
                                \ \alpha_2' \ \alpha_2'' \ \beta_3' \ \beta_3'') , 
                                \\
  \cO^*(\varepsilon_{\alpha_1})  = (\varepsilon_{\alpha_3}
                                \ \varepsilon_{\alpha_2} 
                                \ \varepsilon_{\alpha_1} 
                                ) ,
 \qquad
 \qquad
  \cO^*(\varepsilon_{\beta_1})  = (\varepsilon_{\beta_3}
                                \ \varepsilon_{\beta_2} 
                                \ \varepsilon_{\beta_1} 
                                ) .
\end{gather*}
The weight function $m_{\bullet}^* : \cO(g^*) \to \bN^*$
is given by
$m_{\cO^*(\alpha_1')}^* = m$,
$m_{\cO^*(\varepsilon_{\alpha_1})}^* = 1$,
$m_{\cO^*(\varepsilon_{\beta_1})}^* = 1$.
We define the parameter function $c_{\bullet}^* : \cO(g^*) \to K^*$
to be the constant  function with value $1$.
The weighted triangulation algebra 
$\Lambda^* = \Lambda(Q^*,f^*,m^*_{\bullet},c^*_{\bullet})$
is given by the above quiver $Q^*$ and with 
$18$ commutativity relations
and
$18$ zero-relations,
corresponding to
the six $f^*$-orbits in $Q_1^*$.
For example, we have the relations given 
by the  $f^*$-orbit
$(\alpha_1' \ \varepsilon_{\alpha_3} \ \alpha_2')$:
\begin{align*}
 \alpha_1' \varepsilon_{\alpha_3} 
 &= (
    \beta_1' \beta_1''
    \alpha_2' \alpha_2''
    \beta_3' \ \beta_3''
    \alpha_1' \alpha_1'' 
    \beta_2' \beta_2''
    \alpha_3' \alpha_3'' 
   )^{m-1}
    \beta_1' \beta_1''
    \alpha_2' \alpha_2''
    \beta_3' \ \beta_3''
    \alpha_1' \alpha_1'' 
    \beta_2' \beta_2''
    \alpha_3' 
,
 \\
  \varepsilon_{\alpha_3} \alpha_3'' 
 &= (
    \alpha_1'' 
    \beta_2' \beta_2''
    \alpha_3' \alpha_3'' 
    \beta_1' \beta_1''
    \alpha_2' \alpha_2''
    \beta_3' \ \beta_3''
    \alpha_1' 
   )^{m-1}
    \alpha_1'' 
    \beta_2' \beta_2''
    \alpha_3' \alpha_3'' 
    \beta_1' \beta_1''
    \alpha_2' \alpha_2''
    \beta_3' \ \beta_3''
,
 \\
 \alpha_3'' \alpha_1'
 &= \varepsilon_{\alpha_2} \varepsilon_{\alpha_1}
 , \qquad
 \alpha_1' \varepsilon_{\alpha_3} \varepsilon_{\alpha_2} = 0
 , \qquad
 \varepsilon_{\alpha_3} \alpha_3'' \beta_1' = 0
 , \qquad
 \alpha_3'' \alpha_1' \alpha_1'' = 0
.
\end{align*}
The biserial weighted triangulation algebra
$B = B(Q,f,m_{\bullet})$ 
is then isomorphic to the idempotent algebra $e^* \Lambda^* e^*$,
where $e^*$ is the sum of the primitive idempotents
$e^*_1, e^*_2, e^*_3$ in $\Lambda^*$
corresponding to the vertices $1,2,3$.
\end{example}

We present now an example of an idempotent algebra of a periodic
weighted surface algebra which is neither 
a Brauer graph algebra
nor a weighted surface algebra.

\begin{example}
Let $S$ be a triangle with one puncture,
$T$ the triangulation of $S$
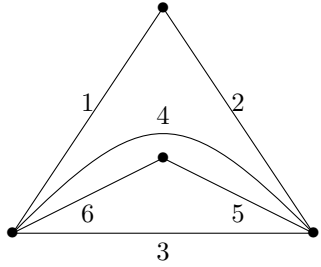
\[
\begin{tikzpicture}
[scale=1]
\node (A) at (-2,0) {$\bullet$};
\node (B) at (2,0) {$\bullet$};
\node (C) at (0,1) {$\bullet$};
\node (D) at (0,3) {$\bullet$};
\coordinate (A) at (-2,0) ;
\coordinate (B) at (2,0) ;
\coordinate (C) at (0,1) ;
\coordinate (D) at (0,3) ;
\draw
(A) edge node [above] {1} (D)
(D) edge node [above] {2} (B)
(A) edge node [below] {6} (C)
(C) edge node [below] {5} (B)
(A) edge [bend left=45,distance=2.5cm] node [above] {4} (B)
(A) edge node [below] {3} (B) ;
\end{tikzpicture}
\]
such that the edges $1,2,3$ are on the boundary,
and let  $\vv{T}$ be 
the orientation of triangles of $T$:
$(1\ 2\ 4)$, $(4\ 5\ 6)$, $(5\ 3\ 6)$.
Then the triangulation quiver
$(Q(S,\vv{T}),f)$ is of the form
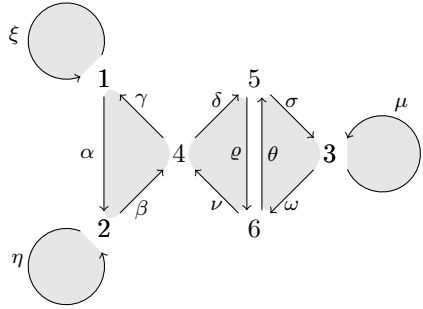
\[
\begin{tikzpicture}
[->,scale=1]
\node (1) at (-2,1) {1};
\node (2) at (-2,-1) {2};
\node (3) at (1,0) {3};
\node (4) at (-1,0) {4};
\node (5) at (0,1) {5};
\node (6) at (0,-1) {6};
\node [circle,minimum size=.5cm](A) at (-2,1) {1};
\node [circle,minimum size=1.cm](B) at (-2.5,1.5) {};
\coordinate  (C1) at (intersection 2 of A and B);
\coordinate  (D1) at (intersection 1 of A and B);
 \tikzAngleOfLine(B)(D1){\AngleStart}
 \tikzAngleOfLine(B)(C1){\AngleEnd}
\fill[gray!20]%
   let \p1 = ($ (B) - (D1) $), \n2 = {veclen(\x1,\y1)}
   in
     (D1) arc (\AngleStart-360:\AngleEnd:\n2); 
\draw%
   let \p1 = ($ (B) - (D1) $), \n2 = {veclen(\x1,\y1)}
   in
     (B) ++(60:\n2) node[below]{\footnotesize\raisebox{-2.5ex}{$\xi\qquad\qquad\qquad\ $}}
     (D1) arc (\AngleStart-360:\AngleEnd:\n2); 
\node [circle,minimum size=.5cm](A) at (-2,-1) {2};
\node [circle,minimum size=1.cm](B) at (-2.5,-1.5) {};
\coordinate  (C) at (intersection 2 of A and B);
\coordinate  (D) at (intersection 1 of A and B);
 \tikzAngleOfLine(B)(D){\AngleStart}
 \tikzAngleOfLine(B)(C){\AngleEnd}
\fill[gray!20]%
   let \p1 = ($ (B) - (D) $), \n2 = {veclen(\x1,\y1)}
   in
     (D) arc (\AngleStart-360:\AngleEnd:\n2); 
\draw%
   let \p1 = ($ (B) - (D) $), \n2 = {veclen(\x1,\y1)}
   in
     (B) ++(60:\n2) node[left]{\footnotesize\raisebox{-7ex}{$\eta\qquad\ $}}
     (D) arc (\AngleStart-360:\AngleEnd:\n2); 
\node [circle,minimum size=.5cm](A) at (1,0) {3};
\node [circle,minimum size=1.cm](B) at (1.7,0) {};
\coordinate  (C) at (intersection 2 of A and B);
\coordinate  (D) at (intersection 1 of A and B);
 \tikzAngleOfLine(B)(D){\AngleStart}
 \tikzAngleOfLine(B)(C){\AngleEnd}
\fill[gray!20]%
   let \p1 = ($ (B) - (D) $), \n2 = {veclen(\x1,\y1)}
   in
     (D) arc (\AngleStart-360:\AngleEnd:\n2); 
\draw%
   let \p1 = ($ (B) - (D) $), \n2 = {veclen(\x1,\y1)}
   in
     (B) ++(60:\n2) node[above]{\footnotesize\raisebox{-7ex}{$\mu$}}
     (D) arc (\AngleStart-360:\AngleEnd:\n2); 
\fill[rounded corners=3mm,fill=gray!20] (-2,1) -- (-2,-1) -- (-1,0) -- cycle;
\fill[rounded corners=3mm,fill=gray!20] (-.1,.9) -- (-.1,-.9) -- (-1,0) -- cycle;
\fill[rounded corners=3mm,fill=gray!20] (.1,.9) -- (.1,-.9) -- (1,0) -- cycle;
\draw (-.1,.75) -- node[left]{\footnotesize$\varrho\!$} (-.1,-.75);
\draw (.1,-.75) -- node[right]{\footnotesize$\!\theta$} (.1,.75);
\draw
(1) edge node[left]{\footnotesize$\alpha$} (2)
(2) edge node[below]{\footnotesize$\beta$} (4)
(3) edge node[below]{\footnotesize$\omega$} (6)
(4) edge node[above]{\footnotesize$\gamma$} (1)
(4) edge node[above]{\footnotesize$\delta$} (5)
(5) edge node[above]{\footnotesize$\sigma$} (3)
(6) edge node[below]{\footnotesize$\nu$} (4);
\end{tikzpicture}
\]
with  $f$-orbits
$(\xi)$,
$(\eta)$,
$(\mu)$,
$(\alpha \ \beta \ \gamma)$,
$(\delta \ \varrho \ \nu)$,
$(\sigma \ \omega \ \theta)$.
Hence we have two $g$-orbits:
\[
  \cO(\alpha)  = (\alpha \ \eta \ \beta \ \delta
                                \ \sigma \ \mu \ \omega \ \nu
                                \ \gamma \ \xi) 
\qquad
\mbox{and}
\qquad
  \cO(\varrho)  = (\varrho \ \theta) . 
\]
Take the weight function $m_{\bullet} : \cO(g) \to \bN^*$
given by
$m_{\cO(\alpha)} = 1$
and
$m_{\cO(\varrho)} = 2$.
Moreover, let $c_{\bullet} : \cO(g) \to K^*$
be the trivial parameter function.
Then the
associated weighted surface algebra 
$\Lambda = \Lambda(S,\vv{T},m_{\bullet},c_{\bullet})$
is given by the quiver $Q(S,\vv{T},)$ and the relations:
\begin{align*}
  \xi^2 &= \alpha \eta \beta \delta \sigma \mu \omega \nu \gamma, 
 &
  \xi^2 \alpha &= 0, 
 &
  \alpha \beta &= \xi \alpha \eta \beta \delta \sigma \mu \omega \nu, 
 &
  \alpha \beta \delta &= 0, 
 &
  \nu \delta &= \theta \varrho \theta,
\\
  \eta^2 &= \beta \delta \sigma \mu \omega \nu \gamma \xi \alpha, 
 &
  \eta^2 \beta &= 0, 
 &
  \beta \gamma &= \eta \beta \delta \sigma \mu \omega \nu \gamma \xi , 
 &
  \beta \gamma \xi &= 0, 
 &
  \nu \delta\sigma  &= 0,
\\
  \mu^2 &= \omega \nu \gamma \xi \alpha \eta \beta \delta \sigma, 
 &
  \mu^2 \omega &= 0, 
 &
  \gamma \alpha &= \delta \sigma \mu \omega \nu \gamma \xi \alpha \eta, 
 &
  \gamma \alpha \eta &= 0, 
 &
  \sigma \omega &= \varrho \theta \varrho,
\\
  \delta \varrho &= \gamma \xi \alpha \eta \beta \delta \sigma \mu \omega, 
 &
  \delta \varrho \theta &= 0, 
 &
  \omega \theta &= \mu \omega \nu \gamma \xi \alpha \eta \beta \delta, 
 &
  \omega \theta \delta &= 0, 
 &
  \sigma \omega \nu &=0, 
\\
  \varrho \nu &= \sigma \mu \omega \nu \gamma \xi \alpha \eta \beta, 
 &
  \varrho \nu \gamma &= 0, 
 &
  \theta \sigma &= \nu \gamma \xi \alpha \eta \beta \delta \sigma \mu, 
 &
  \theta \sigma \mu &= 0. 
 &
\end{align*}
Let $e = e_1 + e_2 + e_3 + e_4$ be the sum of primitive idempotents 
of $\Lambda$ at the vertices $1,2,3,4$,
and $B = e \Lambda e$ the associated idempotent algebra.
Then $B$ is given by the quiver $\Delta$ of the form
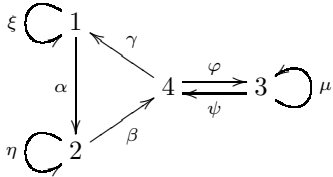
\[
  \xymatrix@R=1pc{
    1
      \ar[dd]_{\alpha}
      \ar@(ul,dl)_{\xi}[] 
  \\ &
    4
      \ar[lu]_{\gamma}
      \ar@<.5ex>[r]^{\varphi}
  &
    3
      \ar@<.5ex>[l]^{\psi}
      \ar@(dr,ur)_{\mu}[]
  \\
    2
      \ar[ru]_{\beta}
      \ar@(ul,dl)_{\eta}[] 
  } 
\]
with the arrows 
$\varphi = \delta \sigma$
and
$\psi = \omega \nu$,
and the induced relations:
\begin{align*}
  \xi^2 &= \alpha \eta \beta \varphi \mu \psi \gamma, 
 &
  \xi^2 \alpha &= 0, 
 &
  \alpha \beta &= \xi \alpha \eta \beta \varphi \mu \psi, 
 &
  \alpha \beta \delta &= 0, 
 &
  \varphi \psi &= 0,
\\
  \eta^2 &= \beta \varphi \mu \psi \gamma \xi \alpha, 
 &
  \eta^2 \beta &= 0, 
 &
  \beta \gamma &= \eta \beta \varphi \mu \psi \gamma \xi , 
 &
  \beta \gamma \xi &= 0, 
 &
  \psi \varphi  &= 0,
\\
  \mu^2 &= \psi \gamma \xi \alpha \eta \beta \varphi, 
 &
  \mu^2 \psi &= 0, 
 &
  \gamma \alpha &= \varphi \mu \psi \gamma \xi \alpha \eta, 
 &
  \gamma \alpha \eta &= 0.
 &
\end{align*}
Then $B$ is not a special biserial algebra, 
and therefore it is not a Brauer graph algebra.
Further, $B$ is not a weighted surface algebra,
because we have  zero-relations 
$\varphi \psi = 0$
and
$\psi \varphi = 0$
of length $2$.
On the other hand, by general theory, the algebra
$B = e \Lambda e$ is  tame and symmetric.
\end{example}

\section{Diagram of algebras}
\label{sec:diagram}

The following diagram shows the relations between
the main classes of algebras occurring in the paper.
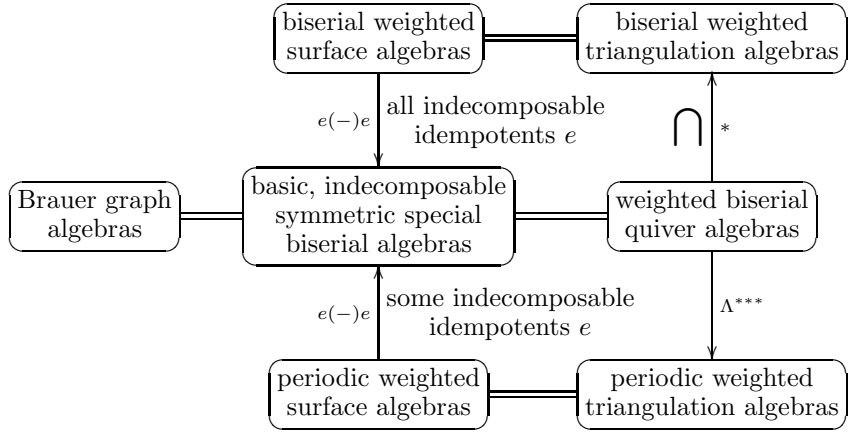
\[
 \xymatrix@R=3pc{
   &
   *+[F-:<5pt>]{ \txt{biserial weighted\\surface algebras}}
   \ar@{=}[r]
   \ar[d]_(.47){e(-)e}^(.47){\txt{all indecomposable\\idempotents $e$}}
   &
   *+[F-:<5pt>]{ \txt{biserial weighted\\triangulation algebras}}
   \\
   *+[F-:<5pt>]{ \txt{Brauer graph\\ algebras}}
   \ar@{=}[r]
   &
   *+[F-:<5pt>]{ \txt{basic, indecomposable\\symmetric special\\biserial algebras}}
   \ar@{=}[r]
   &
   *+[F-:<5pt>]{ \txt{weighted biserial\\quiver algebras}}
    \ar[u]^{\mbox{\LARGE $\bigcap$}}_{*}      
    \ar[d]^{\Lambda^{***}}        
   \\
   &
   *+[F-:<5pt>]{ \txt{periodic weighted\\surface algebras}}
   \ar@{=}[r]
   \ar[u]^(.45){e(-)e}_(.45){\txt{some indecomposable\\idempotents $e$}}
   &
   *+[F-:<5pt>]{ \txt{periodic weighted\\triangulation algebras}}
   ]
 }
\]
where, for a weighted biserial quiver algebra
$B = B(Q,f,m_{\bullet})$,
$B^* = B(Q^*,f^*,m^*_{\bullet})$,
and
$\Lambda^{***} = \Lambda(Q^{***},f^{***},m^{***}_{\bullet}, \mathds{1})$,
with $\mathds{1}$ denoting the trivial weight function
of $(Q^{***},f^{***})$.

\section*{Acknowledgements}

The results of the paper were partially presented during the
Workshop on Brauer Graph Algebras held in Stuttgart in March 2016.
The paper was completed during the visit of the first named author at the 
Faculty of Mathematics and Computer Science
of Nicolaus Copernicus University
in Toru\'n (June 2017).

\end{document}